\documentclass[12pt]{amsart}
\usepackage{epsfig}
\usepackage{amssymb, amsmath}
\usepackage{amsfonts,graphics,epsfig}
\usepackage{mathrsfs}
\usepackage{pb-diagram, pb-xy}
\usepackage[all]{xy}
\usepackage{comment}
\usepackage{color}
\usepackage{tikz-cd}
\usepackage{setspace}
\usepackage{mabliautoref}
\usepackage[left=1in,right=1in,top=1in,bottom=1in]{geometry}
\usepackage{soul}
\usepackage{array}

\newcommand{\sX}{\mathcal{X}}
\newcommand{\sB}{\mathcal{B}}
\newcommand{\sY}{\mathcal{Y}}
\newcommand{\sW}{\mathcal{W}}

\newcommand{\isom}{\cong}
\newcommand{\bQ}{\mathbb{Q}}
\newcommand{\bR}{\mathbb{R}}
\newcommand{\sO}{\mathcal{O}}
\newcommand{\sC}{\mathcal{C}}
\newcommand{\sD}{\mathcal{D}}
\newcommand{\sL}{\mathcal{L}}
\def\mbA{\mathbb{A}}
\newcommand{\mbR}{\mathbb{R}}

\newcommand{\mbZ}{\mathbb{Z}}
\newcommand{\mbQ}{\mathbb{Q}}
\def\mbF{\mathbb{F}}
\def\mbE{\mathbb{E}}
\def\mbN{\mathbb{N}}
\def\mbP{\mathbb{P}}

\newcommand{\<}{\leq}
\def\>{\geq}

\def\ve{\varepsilon}
\def\vphi{\varphi}

\def\subset{\subseteq}

\newcommand{\dis}{\displaystyle}
\newcommand{\lrd}{\lfloor}
\newcommand{\rrd}{\rfloor}

\newcommand{\num}{\equiv}

\newcommand{\bir}{\dashrightarrow}
\def\mcO{\mathcal{O}}

\def\mcC{\mathcal{C}}

\def\mcR{\mathcal{R}}
\def\mcS{\mathcal{S}}

\def\msL{\mathscr{L}}

\def\mfm{\mathfrak{m}}

\def\iff{\Leftrightarrow}

\numberwithin{equation}{section}

\newcommand{\sG}{\mathcal{G}}
\newcommand{\sZ}{\mathcal{Z}}

\def\deg{\operatorname{deg}}
\def\im{\operatorname{Im}}
\def\Supp{\operatorname{Supp}}
\def\dim{\operatorname{dim}}
\def\codim{\operatorname{codim}}

\def\chr{\operatorname{char}}
\def\Ex{\operatorname{Ex}}
\def\lct{\operatorname{lct}}
\def\Spec{\operatorname{Spec}}
\def\Proj{\operatorname{Proj}}
\def\Bl{\operatorname{Bl}}

\def\Pic{\operatorname{Pic}}

\def\NE{\operatorname{NE}}

\def\min{\operatorname{min}}
\def\max{\operatorname{max}}

\def\mult{\operatorname{mult}}
\def\red{\operatorname{red}}
\def\Sing{\operatorname{Sing}}

\def\Diff{\operatorname{Diff}}

\def\Aut{\operatorname{Aut}}

\def\Bs{\operatorname{Bs}}

\author{Omprokash Das}
\address{School of Mathematics\\
	Tata Institute of Fundamental Research\\
	Homi Bhabha Road, Navy Nagar\\
	Colaba, Mumbai 400005}
\email{omdas@math.tifr.res.in}
\email{omprokash@gmail.com}

\author{Joe Waldron}
\address{Department of Mathematics\\
	Michigan State University\\
	619 Red Cedar Road\\
	East Lansing\\
	Michigan 48824\\
	USA}
\email{waldro51@msu.edu}
\date{}

\begin{document}

	\title[$3$-Fold LMMP over imperfect fields of char $p>5$]{On the log minimal model program for $3$-folds over Imperfect Fields of Characteristic $p>5$}
	\maketitle

\begin{abstract}
	We prove that many of the results of the LMMP hold for $3$-folds over  fields of characteristic $p>5$ which are not necessarily perfect.  In particular, the existence of flips, the cone theorem, the contraction theorem for birational extremal rays, and the existence of log minimal models.  As well as pertaining to the geometry of fibrations of relative dimension $3$ over algebraically closed fields, they have applications to tight closure in dimension $4$.
	\end{abstract}
	
	\tableofcontents

\section{Introduction}

 It has recently been shown that most of the log minimal model program  (LMMP) holds for threefolds over an algebraically closed field of characteristic $p>5$, \cite{Kee99,HX15,CTX15,Bir16,BW17} and even if the base field is perfect \cite{GNT15}.  On the other hand, it could be hoped that the LMMP will continue to hold for arbitrary excellent schemes of dimension $3$, as is the case for surfaces \cite{Tan18}.   In this article we prove that many parts of the LMMP hold for threefolds over imperfect fields of characteristic $p>5$.  That is, the schemes we work with are integral, separated and of finite type over a field, but may fail to be geometrically integral or geometrically normal.  
 
 These schemes arise even if one is concerned purely with geometry over algebraically closed fields, because in positive characteristic there are morphisms over algebraically closed fields whose general fibers are excessively singular, for example non-normal or non-reduced.  As a result, many assumptions appearing in the LMMP (e.g. klt) are not preserved by restricting to a general fiber (even after passing to a normalization \cite{PW17}), and so we must restrict to the generic fiber instead (e.g. \cite{BCZ18, Das18,DW18}).  This is a variety over an imperfect field.  This situation is an unavoidable part of an inductive approach to the higher dimensional LMMP in positive characteristic, so we need to be able to apply LMMP results to varieties over imperfect fields.
 
 	Later we will discuss some applications to varieties over algebraically closed fields.  In particular we are able to give the first results about the LMMP in positive characteristic applicable to varieties of dimension greater than $3$.  Furthermore, our results have surprising applications to pure commutative algebra, by work of Aberbach and Polstra \cite{aberbach_polstra}.

 	Our first result is the cone theorem.  Note that the bound on the length of extremal rays takes a somewhat different form when the base field is not algebraically closed.  This new statement is optimal even if the ground field is perfect.

\begin{theorem}\label{thm:cone}
		Let $(X,B)$ be a projective $\bQ$-factorial, dlt pair of dimension at most $3$, over a $F$-finite field $k$ of characteristic $p> 5$. Then there is a countable set of curves $\{C_i\}_{i\in I}$ such that:
		\begin{enumerate}
			\item \[\overline{NE}(X)=\overline{NE}(X)_{K_X+B\geq 0}+\sum_{i\in I}\mathbb{R}_{\geq 0}[C_i]\]
			\item For any ample $\bQ$-divisor $A$, there is a finite subset $I_A\subset I$ such that
			\[\overline{NE}(X)=\overline{NE}(X)_{K_X+B+A\geq 0}+\sum_{i\in I_A}\mathbb{R}_{\geq 0}[C_i]\]
			\item The rays $[C_i]$ do not accumulate in $\overline{NE}(X)_{K_X+B<0}$.
			\item  For each $C_i$, there is a unique positive integer $d_{C_i}$
			depending only on $X, C_i$ and the ground field $k$ satisfying the following properties: 
			\begin{enumerate}
\item			\[ 0<-(K_X+B)\cdot_k C_i\leq  6d_{C_i},\mbox{ and}\]
\item  For any Cartier divisor $L$ on $X$, the integer $L\cdot_k C_i$ is divisible by $d_{C_i}$.
\end{enumerate}
				For a precise description of $d_{C_i}$ see \autoref{sec:fields}.
		\end{enumerate}
	\end{theorem}

\begin{remark}
	In the surface case, the restriction on the characteristic and the ground field is not required, see \autoref{thm:surface_cone}.  We also have a weaker statement for threefolds without these assumptions, see \autoref{thm:cone_theorem}.
	\end{remark}

This new form of the cone theorem turns out to be sufficient for most of the standard applications, and allows us to adapt many of the arguments of \cite{HX15} and \cite{Bir16} to obtain a lot of important results of the LMMP over arbitrary fields:

\begin{theorem}\label{thm:flip}
	Let $(X,B)$ be a quasi-projective $\mbQ$-factorial dlt threefold pair over an arbitrary field of characteristic $p>5$. Let $f:X\to Z$ be a $(K_X+B)$-flipping projective contraction.  Then the flip of $f$ exists.
\end{theorem}

A key part of the proof of the corresponding result over algebraically closed fields \cite{HX15} (which allowed the recent progress on the threefold LMMP to take place) is the use of properties of $F$-singularities to replace vanishing theorems.  These properties make sense only over fields which are $F$-finite, but fortunately we are able to reduce to this from the general case.

Many of the results of \cite{Bir16} depend only on the LMMP, so given the existence of flips, we are able to run the same arguments with some adaptations pertaining to applications of the new cone theorem, to obtain the following results:

\begin{theorem}\label{thm:bpf}
	Let $(X,B)$ be a projective klt threefold pair over an arbitrary field $k$ of characteristic $p>5$ and $B$ is a $\mbQ$-divisor. Let $f:X\to Z$ be a projective contraction, and $D$ a $\mbQ$-Cartier divisor such that both $D$ and $D-(K_X+B)$ are $f$-big and $f$-nef.  Then $D$ is $f$-semi-ample.
	\end{theorem}

\begin{theorem}\label{thm:divisorial-contraction}
	Let $(X,B)$ be quasi-projective  dlt threefold pair over an $F$-finite field $k$ of characteristic $p>5$ with a $(K_X+B)$-negative extremal ray $R$.  Suppose there is a big and nef $\mbQ$-Cartier divisor $D$ such that $D\cdot R=0$.  Then $R$ can be contracted by a projective morphism.
\end{theorem}

\begin{theorem}\label{thm:log-minimal-model}
	Let $(X,B)$ be a quasi-projective klt threefold pair over an $F$-finite field $k$ of characteristic $p>5$, with projective contraction $f:X\to Z$.  If $K_X+B$ is pseudo-effective over $Z$, then $(X,B)$ has a log minimal model over $Z$.
\end{theorem}

\begin{remark}
	The proof of the existence of a Mori fiber space whenever $K_X+B$ is not pseudo-effective in \cite{BW17} relies on the boundedness of $\epsilon$-lc del Pezzo surfaces over algebraically closed fields.  As we have no analogue of this yet for geometrically non-reduced surfaces, the existence of Mori fiber spaces and termination of the  LMMP with scaling remain open.
	\end{remark}

\begin{remark}
	The restriction $p>5$ arises from a generalization of the fact that klt surface singularities are strongly $F$-regular in these characteristics, which is used in the construction of flips in \cite{HX15}, and which fails in low characteristics.    Recently, Hacon and Witaszek have extended parts of the LMMP for threefolds over perfect fields to low characteristics by giving new arguments for the existence of flips.  It is therefore likely that all of our results can be extended to characteristic $5$ using  the construction of flips in \cite{hacon_witaszek_5}, and partially to characteristic $2$ and $3$ using that in \cite{hacon_witaszek_2_3}.  However, we have not attempted to verify this.  
\end{remark}

\begin{remark}
Some of our results are stated for $F$-finite fields.  This covers all realistic geometric applications, for example if $k$ is the function field of a variety over a perfect field, and so $X$ is a generic fiber.  This assumption is necessary due to limitations of the known forms of resolution of singularities.  If projective log resolutions exist over arbitrary fields then this assumption can be removed in all the main results.  While $F$-finite is also needed for some $F$-singularity arguments in the existence of flips, the general case of that result can be deduced from the $F$-finite case, leaving resolution of singularities as the only obstacle.
\end{remark}

We now give some sample applications to higher dimensional varieties over algebraically closed fields.

\begin{corollary}\label{cor:glmm}
	Let $(\sX,\sB)$ be an $n$-dimensional dlt pair over an algebraically closed field of characteristic $p>5$, with a projective contraction $f:\sX\to Z$ to an $(n-3)$-fold $Z$ such that $K_{\sX}+\sB$ is big over $Z$.  Then there is a non-empty open subset $U\subset Z$ such that if $(\sX_U,\sB_U)$ is the induced pair on $f^{-1}(U)$, then $(\sX_U,\sB_U)$ has a good log minimal model over $U$.
	\end{corollary}

\begin{corollary}\label{cor:flip}
	Let $(X,B)$ be an $n$-dimensional dlt pair over an algebraically closed field of characteristic $p>5$, with a projective contraction $f:X\to Z$ to an $(n-3)$-fold $Z$, which is not an isomorphism over the generic point of $Z$.  Suppose that $g:X\to Y$ is a flipping contraction over $Z$ and $h:Y\to Z$ is the induced morphism.  Then there exists a non-empty open subset $U\subset Z$ and a flip $(g|_{X_U})^+:(X_U)^+\to Y_U$ for $g|_{X_U}:X_U\to Y_U$, where $X_U=(h\circ g)^{-1}U$ and $Y_U=h^{-1}U$.
	\end{corollary}

It is possible to formulate analogous versions of many of our other results in higher dimensions, as well as other standard formal consequences of the LMMP, but we leave this to the interested reader.   For example, there are versions of the base point free theorem, existence of crepant dlt models and terminalizations etc.

The next corollary is an example of a formal consequence of the LMMP.    

\begin{corollary}\label{cor:finite_generation}
Let $(X,B)$ be a klt pair which is either a quasi-projective variety of dimension three defined over an arbitrary field $k$ of characteristic $p>5$, or is the localization at a codimension three point of a quasiprojective variety over such a field.  Let $D$ be a 
Weil divisor on $X$.
  Then $\oplus_{m\>0} \sO_X(mD )$ is a finitely generated sheaf of $\sO_X$-algebras.
\end{corollary}

We include this, because it has surprising applications to the  commutative algebra of tight closure, via the following theorem of Aberbach and Polstra  \cite{aberbach_polstra}:

\begin{theorem}\cite[Theorem A]{aberbach_polstra}\label{cor:aberbach_polstra}
	Let $R$ be a $4$-dimensional finitely generated algebra over a field of prime
	characteristic $p > 5$ which has infinite transcendence degree over $\mathbb{F}_p$. If $R$ is weakly
	$F$-regular, then $R$ is strongly $F$-regular.
	\end{theorem}

 As the LMMP for threefolds in positive characteristic has roots in the theory of $F$-singularities, we find it interesting that it now finds applications in the opposite direction, to problems in pure commutative algebra.

\subsection{Organization of the paper}

Due to the way the positive characteristic LMMP works, our results are proved in a somewhat unintuitive order.  \autoref{sec:preliminaries} collects information and techniques we will need, and points out the various difficulties we will encounter when working with arbitrary fields.  In \autoref{sec:fields} we prove some algebraic results necessary to understand the new cone theorem. In \autoref{sec:keel} we prove weak versions of the cone and base point free theorems similar to those of \cite{Kee99}.  In \autoref{sec:complements} and \autoref{sec:flips} we prove the existence of flips, log minimal models and the base point free theorem for big divisors. More specifically, the proofs of \autoref{thm:flip}, \ref{thm:bpf}, \ref{thm:divisorial-contraction} and \ref{thm:log-minimal-model} appears at the end of \autoref{sec:flips}. In \autoref{sec:cone2} we prove the full version of the cone theorem for threefolds. Finally, in \autoref{sec:corollary} we prove the \autoref{cor:glmm}, \ref{cor:flip}, \ref{cor:finite_generation} and \ref{cor:aberbach_polstra}.

\subsection{Acknowledgments}
The authors would like to thank  Christopher Hacon, Takumi Murayama, Zsolt Patakfalvi, Marta Pieropan, Thomas Polstra and Burt Totaro for helpful conversations, and Nivedita Bhaskar and Lena Ji for carefully reading early versions of the paper.

This material is partly based upon work supported by the National Science Foundation under Grant No. 1440140, while the second author was in residence at the Mathematical Sciences Research Institute in Berkeley, California, during the Spring of 2019.  Waldron is supported by a grant from the Simons Foundation (850684).

\section{Preliminaries}\label{sec:preliminaries}

\subsection{Varieties}

\begin{definition}
A variety $X$ over a field $k$ is an integral scheme which is separated and of finite type over $k$.
 If $X$ has dimension $1$ (resp. $2$ or $3$) we call it a \emph{curve} (resp. \emph{surface} or \emph{threefold}).  
 \end{definition}

\begin{remark}
	 We adopt the definition of variety as in \cite[Definition 020D]{STP}, in which $X$ \emph{need not be geometrically integral over $k$}.  
	 There is some disagreement among the experts over whether varieties should be  geometrically integral over the base field or not, and unfortunately both conventions are in use, sometimes without comment. 
	 
	 We would like to advocate for adoption of the above definition for several reasons:
	 \begin{enumerate}
	 	\item Our definition of ``variety over a field $k$'' is stable under passing to the generic fiber of a fibration, and to integral subschemes.  (Of course this is at the cost of not being stable under field extensions.)
	 	\item When integral schemes which are separated and of finite type over a non-algebraically closed field arise in positive characteristic geometry, it is usually precisely because some hypotheses are not preserved under field extension.
	 	\item \cite{STP} is becoming the go-to reference for checking this type of definition, so using its convention will be less confusing to future readers, especially where EGA provides no alternative. \cite[Chapter II.4]{Har77} defines ``variety'' only under the assumption that the ground field is algebraically closed.
	 	\item If we were to replace ``threefold'' with ``integral scheme of dimension three separated and of finite type over a field'' everywhere it appears, this article would probably be at least a page longer, and unreadable.
	 	\end{enumerate}
	\end{remark}

\subsubsection*{Bertini's theorems}
Let $X$ be a variety over an \emph{infinite} field $k$. Let $D$ be a Cartier divisor on $X$ and $V=H^0(X, \mcO_X(D))$. Let $W$ be a $k$-subspace of $V$. Let $\Gamma=\mbP(W)\subseteq |D|=\mbP(V)$ be a linear system on $X$, $\Bs(\Gamma)$ be the base locus of $\Gamma$ and $\Phi_\Gamma:X\bir\mbP^n_k$ be the rational map defined by $\Gamma$, where $n=\dim_k V$. 
\begin{theorem}[First Bertini theorem]\label{thm:bertini-i}\cite[Theorem 3.4.10]{FCG}
With the notations above, if 
$\dim\Phi_\Gamma(X)>1$ and $\codim_X\Bs(\Gamma)\>2$, then a general member of $\Gamma$ is \emph{irreducible}. In particular, if $X\subseteq\mbP^n_k$ is a quasi-projective (irreducible) variety with $\dim X\>2$, then for a general hyperplane $H\subseteq\mbP^n_k$, the intersection $X\cap H$ is \emph{irreducible}.\\ 	
\end{theorem}

\begin{theorem}[Second Bertini Theorem]\label{thm:bertini-ii}\cite[Corollary 3.4.14]{FCG}
	Let $X \subseteq \mathbb{P}^n_k$ be a quasi-projective scheme of finite type over an infinite field $k$. If $X$ is regular (resp. normal, resp. reduced, resp. satisfies $R_\ell$ for some $\ell\>0$), then for a general hyperplane $H\subseteq \mathbb{P}^n_k$, the intersection $X\cap H$ is also regular (resp. normal, resp. reduced, resp. satisfies $R_\ell$ for some $\ell\>0$).\\
\end{theorem}

\begin{proof}
	This is stated in \cite[Corollary 3.4.14]{FCG} for $X$ projective. However, their proof is based on local computations,
	 which hold equally well for quasi-projective schemes.
	\end{proof}

Because of these results, we will usually assume that $k$ is infinite.  We lose no generality from this assumption as all of our main results were proved for perfect (and hence for finite) fields in \cite{GNT15, HNT17nov}.

\subsection{$F$-finite fields}

There is a particular class of imperfect fields which is more convenient to work with: those which are $F$-finite.  This says roughly that $k$ is similar to a field which arises from geometry.  For example, a perfect field is $F$-finite and any function field of a variety over an $F$-finite field is $F$-finite.
The reasons we need to work with this class of field are explained below.

\begin{definition}
	A field $k$ is $F$-finite if the Frobenius $F:\Spec(k)\to\Spec(k)$ is finite, or equivalently $[k:k^p]<\infty$, where $k^p=\{a^p: a\in k\}$.
	\end{definition}

\begin{lemma}\label{lem:F-finite-vs-differentially-finite}
	Let $k$ be a field of characteristic $p>0$.  Then the following are equivalent:
	\begin{enumerate}
		\item $k$ is $F$-finite.
		\item  $k$ is differentially finite over the perfect subfield $k_0$, i.e., $\dim_k\Omega_{k/k_0}<\infty$, where $k_0:=\cap_{n\>1} k^{p^n}\subset k$ and $k^{p^n}:=\{a^{p^n}: a\in k\}$.
	\end{enumerate}
\end{lemma}
\begin{proof}
	  Clearly $k_0$ is perfect, and in fact it is the largest perfect subfield of $k$. Note that the compositum of the fields $k_0$ and $k^p$ is $k^p$, since $k_0\subset k^p$. Thus from \cite[Definition 07P1]{STP} it follows that $k$ is $F$-finite if and only if $k$ has a finite $p$-basis over $k_0$. But then this is equivalent to  $\dim_k\Omega_{k/k_0}<\infty$ by \cite[Lemma 07P2]{STP}.
\end{proof}

\subsubsection{Resolution of singularities}
Resolution of singularities was proved in \cite{Cut09}, \cite{CP08} and \cite{CP09} for threefolds over a field $K$ which is differentially finite over some perfect subfield $k\subset K$, which we see from  \autoref{lem:F-finite-vs-differentially-finite} is the same as assuming that $K$ is $F$-finite.

\begin{theorem}\cite{CP09}\cite{CJS09}
	Let $X$ be a reduced quasi-projective scheme of dimension at most $3$ over an $F$-finite field, and $B$ a reduced Weil divisor. Then there exists a projective birational morphism $\pi:Y\to X$ such that 
	\begin{enumerate}
		\item $Y$ is regular and it is obtained via successive blow up of regular centers,
		\item $\pi$ is an isomorphism over the snc locus of $(X,B)$, and
		\item $\pi^{-1}(\Sing(X)\cup \Supp(B))$ is a divisor with snc support.
	\end{enumerate}
\end{theorem}

In this case $\pi:Y\to X$ is  called a \emph{log resolution} of the pair $(X,B)$.

\begin{remark}\label{rmk:on-projective-resolution}
	Resolution of singularities was extended to arbitrary quasi-excellent Noetherian schemes of dimension at most $3$ (in particular threefolds over arbitrary fields) in  \cite{CP19} and \cite{CJS09}.  However, the latter produces a resolution which may not be a projective morphism.  This is not enough to run certain LMMP arguments such as the Shokurov pl-flip reduction, therefore we must work with $F$-finite fields.  If resolution by a projective morphism were known over an arbitrary field, all of our results would immediately generalize to that setting.
	\end{remark}

\subsubsection{$F$-singularities}

The following linear system has played a key role in the recent progress in birational geometry in positive characteristic, in particular in \cite{HX15}.

\begin{definition}\cite[Definition 4.1]{Sch14}
	Let $(X,\Delta\>0)$ be a log pair with $\bQ$-boundary over an $F$-finite field $k$, and $M$ a line bundle on $X$.  Then for any integer $e>0$,
	\[S^0\left(X,\sigma\left(X,\Delta\right)\otimes M\right):=\bigcap_{n>0} \im\left(H^0\left(X,F^{ne}_*\left(\sigma\left(X,\Delta\right)\otimes \mathcal{L}_{ne,\Delta}\otimes M^{p^{ne}}\right)\right)\to H^0\left(X,\sigma\left(X,\Delta\right)\otimes M\right)\right)\]
	Here $\sigma(X,\Delta)$ is the non-$F$-pure ideal of $(X,\Delta)$ (see \cite[Definition 2.7]{Sch14}), and $\sL_{ne,\Delta}=\sO_X((1-p^{ne})(K_X+\Delta))$.
\end{definition}

Most of the results concerning this linear system and its properties already hold in the generality of a variety over an $F$-finite field, with this assumption being necessary to apply Grothendieck duality for a finite morphism to the Frobenius.  These results can all be found in Section 2 of \cite{HX15}, and the proofs of these results referred to by \cite{HX15} all hold over $F$-finite fields.

\subsection{Keel's semi-ampleness criterion}

We use the following semi-ampleness criterion due to Keel, which holds for arbitrary schemes which are projective over a field of positive characteristic.

\begin{definition}\label{def:nef-exceptional-locus}
	Let $X$ be a scheme proper over a field $k$ and $L$ a nef line bundle on $X$. A subvariety $Z\subset X$ is called \emph{$L$-exceptional} if $L|_Z$ is not big, i.e., $(c_1(L))^{\dim Z}\cdot Z=0$. The \emph{exceptional locus} of $L$, denoted by $\mathbb{E}(L)$ is defined as the closure of the union of all $L$-exceptional subvarieties, with reduced induced structure. 
\end{definition}

\begin{definition}\label{def:EWM}
	Let $X$ be a proper scheme defined over an arbitrary field $k$ and $L$ a nef line bundle. Then $L$ is called \emph{Endowed With a Morphism} or a \emph{EWM} if there exists a proper morphism $f:X\to Y$ to a proper algebraic space $Y$ such that $f$ contracts exactly the $L$-exceptional subvarieties of $X$.
\end{definition}

\begin{theorem}\cite[Theorem 1.9]{Kee99}
	Let $X$ be a scheme projective   over  field of positive characteristic and $L$ be a nef line bundle on $X$.  Then $L$ is semi-ample (resp. EWM) if and only if $L|_{\mathbb{E}(L)}$ is semi-ample (resp. EWM).
\end{theorem}

Keel also proved versions of the cone and contraction theorems, however his proofs of these hold only over algebraically closed fields.  We shall adapt and recover these results over imperfect fields in \autoref{sec:keel}.

\subsection{Log minimal model program}

We will follow \cite{Kol13} for definitions of pairs and singularity classes, except that our boundaries will have $\bR$-coefficients unless stated otherwise.

We refer to \cite{Bir16} for definitions such as \emph{log birational model}, \emph{log minimal model} etc.

\subsubsection{LMMP for Surfaces}

Almost all results on the LMMP for surfaces hold over imperfect fields by \cite{Tan18}.   A notable exception is part of the cone theorem, which we deal with later. 

\subsubsection{Adjunction and DCC coefficients} A key assumption in the proof of existence of flips in \cite{HX15} was that the pair has standard coefficients.  We need to use the same condition.

\begin{definition}\label{def:standard-set}
	Let $\mathcal{S}:=\{1-\frac{1}{n}:n\in\mathbb{N}\}\cup\{1\}$.  We say that a pair $(X, \Delta)$ whose boundary $\Delta$ has coefficients in $\mathcal{S}$ has \emph{standard coefficients}.
	\end{definition}

\begin{proposition}\cite[Def. 4.2]{Kol13}\cite[Pro. 4.2]{Bir16}\cite[Pro. 2.8]{Das18}
	Let $(X,B)$ be a log canonical pair with boundary coefficients in a DCC set $\mathcal{J}$,  and $S$ a component of $\lfloor B\rfloor$.  Let $B_{S^n}$ be the different defined by the adjunction $K_{S^n}+B_{S^n}=(K_X+B)|_{S^n}$, where $S^n\to S$ is the normalization. Then the coefficients of $B_{S^n}$ are contained in a fixed DCC set $D(\mathcal{J})$. If $\mathcal{J}=\mathcal{S}$, the set of standard coefficients, then $D(\mathcal{J})=\mathcal{S}$.
\end{proposition}

\subsubsection{Negativity Lemma}

\begin{lemma}[Negativity Lemma]\label{lem:negativity-lemma}
	Let $f:X\to Y$ be a proper birational morphism between two normal varieties defined over some arbitrary field $k$. Let $B$ be a $\mbQ$-Cartier divisor on $X$ such that $-B$ is $f$-nef. Then $B$ is effective if and only if $f_*B$ is effective.	
\end{lemma}

\begin{proof}
	If $k$ is finite, pass to $\overline{k}$, noting that effectivity and nefness are preserved under field extension.  The same proof as in \cite[Lemma 3.39]{KM98} works using the Hodge index theorem for excellent surfaces \cite[Theorem 10.1]{Kol13}, relying on the Bertini theorems in \autoref{thm:bertini-i} and \autoref{thm:bertini-ii} to cut down the dimension.
\end{proof}

\section{Field extensions}\label{sec:fields}

Let $X$ be a projective variety over a field $k$ of characteristic $p>0$.  Fix an algebraic closure  $\overline{k}$ of $k$ throughout. As a scheme, $X$ comes with a natural choice of ground field, which is $k_X:=H^0(X,\sO_X)$. This is the unique largest field contained in $\sO_X$, i.e., it is the largest field over which $X$ is a $k_X$-scheme.  Furthermore, $k_X$ is the normalization of $k$ in the function field $K(X)$, in particular, $k_X/k$ is a finite extension, and to define $k_X$ this way does not require $X$ to be projective. 

In some situations it is not convenient to consider a variety to have ground field $k_X$.  For instance, if $X$ is a variety with $k_X=H^0(X,\sO_X)$, and $Z$ a reduced subscheme with $k_Z=H^0(Z,\sO_Z)$, then it is possible for $k_Z$ to be a non-trivial  finite extension of $k_X$, but we want to consider both to have the same ground field.

We would like to point out that geometric properties such as geometrically reduced, smooth etc. depend heavily on the choice of the ground field. It is also important to keep track of the choice of ground field in intersection theory, because intersection numbers depend on the choice as follows:

\begin{definition}\cite[7.3.1]{Liu02},\cite{Ful98}
	Let $C$ be a curve over a field $k$, and $D$ be a Cartier divisor on $C$.  Then 	
	if we express $D$ as $D=\sum_{x\in C} v_x(D)x$, we define $$\deg_kD=\sum_{x\in C} v_x(D)[k(x):k].$$

	Now let $X$ be a normal variety over a field $k$, containing a curve $C$ and Cartier divisor $D$.  Then there is an naturally defined intersection $0$-cycle $D\cdot C=D|_C$, given as a subscheme of $X$.
	The \emph{intersection number over $k$} is $$D\cdot_kC =\mathrm{length}_k(\sO_{D\cdot C})=\deg_k D|_C.$$ 
\end{definition}

Note that these numbers depend  on the choice of ground field, while the cycles do not.
This means one must be careful of the effect of the ground field, for example, on Riemann-Roch for regular curves which takes the following form:
\[\chi_k(X, \mcO_X(D))=\deg_kD-\dim_kH^1(C, \sO_C)+\dim_kH^0(C,\sO_C).\]

The notion of ground field should not be confused with \emph{field of definition} described below.

\begin{definition}\cite[4.8.11]{EGAIV2}\label{def:field_definition}
	Let $X$ be a variety over a field $k$.  Let $Z\subset X_{K}$ be a closed subscheme (possibly non-reduced or non-irreducible), where $K$ is an algebraic extension of $k$ and $X_{K}=X\otimes_kK$.  
	
	For a field $k\subset F\subset K$, we say that $Z$  is \emph{defined over $F$} if there is a subscheme $Z_F$ of $X_F=X\otimes_k F$ such that $Z=Z_F\otimes_F K$.  
	If this holds we say that $F$ is a \emph{field of definition} of $Z$ with respect to $X$.
	
	Furthermore, there is a \emph{unique} subfield $L\subset K$ called the \emph{minimal field of definition} of $Z$ in $X_{K}$, which is contained in every field of definition of $Z$ with respect to $X$. We note that $L/k$ is a finite extension by \cite[4.8.13]{EGAIV2}.
\end{definition}

\begin{remark}
	\begin{itemize}
		\item Whether or not $Z$ is defined over $F$ is a property of the embedding $i:Z\hookrightarrow X\otimes_kK$ rather than the set $Z$ alone.  We suppress some of this if $X$, $k$ and $K$ are clear from the context.
		\item 
	Let $Z\subseteq X_{\overline{k}}$ be a closed subscheme with $F$ a field of definition. If $Z$ is reduced (resp. irreducible, resp. integral), then it follows from the definition above that  $Z_F$ is geometrically reduced (resp. geometrically irreducible, resp. geometrically integral) over $F$.
	\end{itemize}
\end{remark}

\begin{lemma}\label{lem:splitting_degree}
	Let $X$ be a variety over $k$. Fix an algebraic closure $\overline{k}$ of $k$. Let $C\subseteq X$ be an irreducible closed subvariety and $\{C_i\}_{i\in I}$ be the irreducible components of $C\otimes_k \overline{k}$ with reduced induced structure.  For each $i$, let $L_i\subseteq \overline{k}$ be the minimal field of definition of $C_i$ in $X\otimes_k\overline{k}$. 
	
Then the group $\Aut(\overline{k}/k)$ acts transitively on the fields $\{L_i\}_{i\in I}$, which are pairwise isomorphic under this action.  Furthermore, if we let $C^{L_i}$ denote the component of $(C\otimes_kL_i)_{\mathrm{red}}$ which satisfies $C^{L_i}\otimes_{L_i}\overline{k}=C_i$, then for each $i,j$ there is a commutative diagram with isomorphic rows:
	\begin{center}
			\begin{tikzcd}
			C^{L_i}\ar{r}\ar{d} & C^{L_j}\ar{d}\\
			\Spec(L_i) \ar{r}& \Spec(L_j)
		\end{tikzcd}
	\end{center}
\end{lemma}

\begin{proof}
	Let $F$ be the separable closure of $k$ in $\overline{k}$. Then every element $\phi$ of $\mathrm{Gal}(F/k)$ acts on $X\otimes_k F$ by $\mathrm{id}\otimes\phi$.  It therefore induces an action on $C\otimes_k F$ and its irreducible components. By \cite[Lemma 04KY]{STP}, this action on the irreducible components of $C\otimes_k F$ is transitive, as $C$ is irreducible.  Since $\overline{k}/F$ is purely inseparable, the irreducible components of $C\otimes_k F$ are in bijective correspondence with those of $C\otimes_k\overline{k}$.
	
	Now by the uniqueness of $p^{\mathrm{th}}$ root in characteristic $p>0$, any automorphism $\phi:F\to F$ extends uniquely to an automorphism $\phi_{\overline{k}}:\overline{k}\to\overline{k}$ as follows
	\[\phi_{\overline{k}}(a):=\left(\phi\left(a^{p^e}\right)\right)^{\frac{1}{p^e}},\mbox{ where } a^{p^e}\in F\mbox{ for some } e\gg0, \quad\mbox{for all } a\in\overline{k}.\]
	The uniqueness of $p^{\mathrm{th}}$ root ensures that this does not depend on the choice of $e$.  Furthermore, one can verify that this is a ring homomorphism using the identity $(a+b)^p=a^p+b^p$.

	Fix two irreducible components $C_j$ and $C_i$ of $(C\otimes_k\overline{k})_{\mathrm{red}}$, and let $L_j$ be the minimal field of definition of $C_j$ in $X\otimes_k\overline{k}$.
	\begin{claim}\label{clm:relating-C_i-and-C_j}
		There is a field $L_j'\subseteq\overline{k}$ which is a field of definition of $C_i$ such that $[L_j':k]=[L_j:k]$.
	\end{claim}

	\begin{proof}[Proof of \autoref{clm:relating-C_i-and-C_j}]
		
		Since the action of $\mathrm{Gal}(F/k)$ permutes the irreducible components of $C\otimes_k F$ and the irreducible components of $C\otimes_k F$ and $C\otimes_k\overline{k}$ are in bijective correspondence, there exists a $\phi\in\mathrm{Gal}(F/k)$ such that its extension $\phi_{\overline{k}}\in\mathrm{Aut}(\overline{k}/k)$ satisfies $\phi_{\overline{k}}(C_j)=C_i$. 	Let $L_j'=\phi_{\overline{k}}(L_j)$.  We consider the following commutative diagram:
		\begin{center}
			\begin{tikzcd}
				C_j\subset C\otimes_k\overline{k}\ar{rr}{f\;=\;\mathrm{id}\otimes\phi_{\overline{k}}}\ar{d}& & C_i\subset C\otimes_k\overline{k}\ar{d}\\
				C^{L_j}\subset C\otimes_k L_j \ar{rr}{g\;=\;\mathrm{id}\otimes\phi_{\overline{k}}|_{L_j}}\ar{dr}& & C^{L_j'}\subset C\otimes_k{L_j'}\ar{dl}\\
				& C  &
			\end{tikzcd}
		\end{center}
		Since $L_j$ is the minimal field of definition of $C_j$, there is an irreducible component (with reduced structure) $C^{L_j}$  of $C\otimes_k L_j$ such that $C^{L_j}\otimes_{L_j}\overline{k}=C_j$. Observe that, $f$ is an isomorphism and $f^{-1}(C_i)=C_j$, since $\phi(C_j)=C_i$. Since $g$ is also an isomorphism, from the commutativity of the above diagram it follows that there is an irreducible component  $C^{L_j'}$ of $C\otimes_kL_j'$ which is geometrically integral and such that $C^{L_j'}\otimes_{L_j'} \overline{k}=C_i$. Thus $L_j'$ is a field of definition of $C_i$, and we also have $[L_j':k]=[L_j:k]$, since $L_j$ and $L_j'$ are isomorphic  via $\phi_{\overline{k}}$.\\
	\end{proof}	  
	
	We now return to the main proof. Since $L_i$ is the minimal field of definition of $C_i$, we have  $L_i\subseteq L_j'$ and so we see that $[L_j:k]\geq [L_i:k]$.  But in fact, we may run the argument with $i$ and $j$ swapped, yielding $[L_j:k]=[L_i:k]$. From the uniqueness of minimal field of definition it also follows that $L_i\cong L_j'$, and thus the morphism $\phi_{\overline{k}}|_{L_j}$ yields an isomorphism $L_j\cong L_i$. But then $C^{L_j'}$ which appeared in the proof of the claim was in fact isomorphic to $C^{L_i}$, and so we have obtained the isomorphism in the statement.\\
\end{proof}

\begin{lemma}\label{lem:integer}
	In the notation of \autoref{lem:splitting_degree}, let $L$ be the compositum of $\{L_i\}_{i\in I}$ in $\overline{k}$.  Let $E$ be an irreducible component of $(C\otimes_k L)_{\mathrm{red}}$.  Then the quantity \[d_C:=\frac{[L:k]}{[K(E):K(C)]}\] is an integer.
\end{lemma}

\begin{proof}
	Note that $L/k$ is a normal field extension, since the image of any injection $L\to\overline{k}$ is the compositum of the images of the $L_i$. But the images of the $L_i$ are these same fields permuted, since $\{L_i\}$ form a complete set of Galois conjugates by \autoref{lem:splitting_degree} and \cite[V.3.8]{hungerford_algebra}.

	 Let $H^0(C,\sO_C)=k_C$.  We first claim that we may assume that $k=k_C$. 
	It follows from the construction that $E$ is geometrically integral over  $L$ since $L$ is a field of definition for every component of $C\otimes_k\overline{k}$, and consequently, $H^0(E, \mcO_E)=L$ by \cite[Lemma 0BUG]{STP}.
	The composition $E\to C\to k$ factors through the Stein factorization $C\to k_C$, and so $k_C\subset H^0(E,\sO_E)=L$, and we obtain the following diagram:
	
	\[
	\xymatrix{
		E\ar[r]&C \ar@{-->}[dr]^\psi\ar@/^2pc/[drr]^{\cong}\ar@/_2pc/[ddr]& &\\
		&	&C\otimes_k k_C\subset X\otimes_k k_C \ar[r]^-\phi\ar[d]& C\subset X\ar[d]\\
		&	& \Spec k_C \ar[r]& \Spec k
	}
	\]
	
	As the top diagonal arrow is an isomorphism, the image of $C$ under the induced map is isomorphic to $C$.  While $C\otimes_{k} k_C$ potentially has several (non-reduced) components, the map $E\to\Spec(k_C)$ which is fixed from the start picks out a single component of interest to us, so we are free to ignore the rest.   Since $(X\otimes_k k_C)\otimes_{k_C}\overline{k}\isom X\otimes_k \overline{k}$, it follows that the minimal fields of definition  are the same whether it is calculated relative to $C\subset X$ or  $C'\subset C\otimes_kk_C\subset X\otimes_k k_C$, {where $C'\isom C$ is that particular component described before}. Note that $C\to X\otimes_k k_C$ is a closed embedding by \cite[Lemma 03BB]{STP}, since $\phi\circ \psi$ is an isomorphism.  Observe that $[K(E):K(C)]$ is unaffected under this consideration, so if we obtain integrality of  \[\frac{[L:k_C]}{[K(E):K(C)]}\] then the same for $k$ by multiplying  by $[k_C:k]$.  
	
	We next show that we may assume that $L/k_C$ is separable.   By \cite[Corollary 3.6.5]{roman_field_theory}, since $L/k_C$ is normal, there exists a purely inseparable extension  $k_{in}/k_C$, such that $L/k_{in}$ is separable.
	 Furthermore, if $k_{sep}$ is the separable closure of $k_C$ in $L$, we have $L=k_{in}\cdot k_{sep}$.  Now let $C_{in}=(C\otimes_{k_C}k_{in})_{red}$, where $C_{in}$ is irreducible since $k_{in}/k_C$ is purely inseparable.  Since $C_{in}$ is irreducible, the irreducible components of $C_{in}\otimes_{k_{in}}L$ biject with the irreducible components of $C\otimes_{k_C} L$. Also, $K(C)\otimes_k k_{in}$ is a local ring of dimension $0$ with residue field $K(C_{in})$.  By the Cohen structure theorem, there is a coefficient ring, i.e. a subring $F$ of $K(C)\otimes_k k_{in}$ which maps isomorphically to the residue field $K(C_{in})$.
	 
	 Therefore we have \[[k_{in}:k_C]=\dim_{K(C)}(K(C)\otimes_k k_{in})=[K(C_{in}):K(C)]\dim_{F}(K(C)\otimes_k k_{in})\]
	 
	 On the other hand the tower law gives \[\frac{[L:k_C]}{[K(E):K(C)]}=\frac{[L:k_{in}][k_{in}:k_C]}{[K(E):K(C_{in})][K(C_{in}):K(C)]}\]
	 Therefore we are done if we show integrality of 
	 \[\frac{[L:k_{in}]}{[K(E):K(C_{in})]}.\]
	 
	$L/k_{in}$ is separable by construction of $k_{in}$, and we claim that $K(E)/K(C_{in})$ is also separable.  One way to see this 
	is to note that $K(E)/K(C_{in})$ is separable if and only if {$K(E)\otimes_{K(C_{in})} K(E)$ is reduced}, and this is a quotient of the Artinian ring \[(L\otimes_{k_{in}}K(C_{in}))\otimes_{K(C_{in})}(L\otimes_{k_{in}}K(C_{in}))\cong L\otimes_{k_{in}}L\otimes_{k_{in}}K(C_{in})\]
	This in turn is reduced since $K(C_{in})$ is reduced and $L/k_{in}$ is  separable.
	Hence both extensions are separable.  We also have $H^0(C_{in},\sO_{C_{in}})=k_{in}$ by \cite[Lemma 5.5]{ji_waldron}. 
	 
	 Let $C_1,...,C_k$ be the distinct components of $C_{in}\otimes_{k_{in}}L$.  By flat base change we have 
	 \[[L:k_{in}]=\dim_{K(C_{in})}(K({C_{in}})\otimes_{k_{in}}L)\]  But since $L/{k_{in}}$ is separable, $K(C_{in})\otimes_{k_{in}}L$ is a reduced Artinian ring, and hence is a direct sum of fields. But this is the fiber of $C_{in}\times_{\Spec(k_{in})}\Spec(L)$ over the generic point of $C_{in}$, and hence the fields in question are the function fields of the $C_{i}$.
	 
	  Hence we have
	\[[L:k_{in}]=\sum_{i=1}^k[K(C_i):K(C_{in})].\]
\autoref{lem:splitting_degree} says that the action of $\Aut(\overline{k}/k_{in})$ on the $C_i$ is transitive.  This action is preserved under the surjection $\Aut(\overline{k}/k_{in})\twoheadrightarrow\Aut(L/k_{in})$, and so $\Aut(L/k_{in})$ acts transitively on the $C_i$,  and hence also on their function fields.  Therefore each of the $K(C_i)$ are isomorphic to one another via this action, and $[K(C_i):K(C_{in})]$ are equal since $K(C_{in})$ is fixed by the action of $\Aut(L/k_{in})$.
  Hence the summands are all equal, and equal to $[K(E):K(C_{in})]$ since $E$ is one of the $C_i$.
	This completes the proof.
	 
	\end{proof}

\section{Cone theorem I and special LMMPS}\label{sec:keel}

	It was shown in \cite[Example 7.3]{Tan15f} that the usual bound on the length of extremal rays fails on a variety defined over a non-algebraically closed ground field. In this section we introduce a correction term into the bound which suffices for many applications.

	\begin{lemma}\label{lem:constant}
	Let $X$ be a proper variety over a  field $k$ and denote $\phi:X\otimes_k\overline{k}\to X$.  Let $C$ be a curve on $X$ and  $C^{\overline{k}}$ be an integral curve in $X\otimes_k{\overline{k}}$ whose image is $C$.    If $d_C$ is the positive integer which appears in \autoref{lem:integer},  then for any $\bR$-Cartier divisor $D$ we have:
	\[D \cdot_k C=d_C (\phi^*D\cdot_{\overline{k}} C^{\overline{k}}).\]
	In particular, if $D$ is a Cartier divisor on $X$, then $D\cdot_k C$ is divisible by $d_C$. 
		\end{lemma}
	\begin{proof}
		As the constant $d_C$ does not depend on $D$, it is enough to prove the equality for $D$ a Cartier divisor.		
		Let $L$ be the compositum of the minimal fields of definitions of the components of $(C\otimes_k\overline{k})_{\mathrm{red}}$, and $C^{L}$ be the 
		 irreducible component (with reduced structure) of $C\otimes_kL\subseteq X\otimes_kL$ such that $C^{L}\otimes_L\overline{k}= C^{\overline{k}}$. Let $\theta:X_{\overline{k}}\to X_L:=X\otimes_k L$ and $\phi:X_L\to X$ be the projection morphisms. Note that $\phi_*C^L$ is a cycle of dimension $1$ in $X$ supported on $C$; in fact, $\phi_*C^L=[K(C^L):K(C)]C$, where $K(C^L)$ and $K(C)$ are the function fields of $C^L$ and $C$, respectively.

		Now we have
			\begin{align*}
			D\cdot_k C &=\frac{1}{[K(C^L):K(C)]} D\cdot_k \phi_*C^L \\
			&= \frac{1}{[K(C^L):K(C)]} \phi^*D\cdot_k C^L \\
			&=
			\frac{[L:k]}{[K(C^L):K(C)]}\phi^*D\cdot_L C^L\\
			&=	\frac{[L:k]}{[K(C^L):K(C)]}\deg_L((\phi^*D)|_{C^L})\\
			&=	\frac{[L:k]}{[K(C^L):K(C)]}\deg_{\overline{k}}(((\phi\circ\theta)^*D)|_{C^{\overline{k}}})\quad\mbox{by \cite[Pro. 3.7(a), Ch. 7]{Liu02}} \\
			&=	\frac{[L:k]}{[K(C^L):K(C)]}((\phi\circ\theta)^*D)\cdot_{\overline{k}}C^{\overline{k}}\\ 
			&=d_C(\pi^*D\cdot_{\overline{k}}C^{\overline{k}})\quad 
			\end{align*}
		
		\end{proof}
	
\begin{lemma}\label{lem:inequality}
	Let $X$ be a proper variety over a  field $k$, $C$ a curve on $X$ and $Y$ is the normalization of $X\otimes_k \overline{k}$, with  $\pi:Y\to X$ the induced morphism.
	
	Let $C^Y$ be a curve on $Y$ such that $\pi(C^Y)=C$.  If $d_C$ is the number which appears in \autoref{lem:integer},  then for any $\bR$-Cartier divisor $D$ with $0<D\cdot_k C$ we have:
	\[0<D\cdot_k C\leq d_C (\pi^*D\cdot_{\overline{k}}C^{Y})\]
\end{lemma}
\begin{proof}
	Let $\phi$ be the natural map $X\otimes_k\overline{k}\to X$ and $C^{\overline{k}}$ be an integral curve in $X\otimes_k\overline{k}$ such that $\psi(C^Y)=C^{\overline{k}}$ and $\phi(C^{\overline{k}})=C$, where $\psi:Y\to X\otimes_k\overline{k}$ is the induced morphism.\\
	Then $\psi_*(C^Y)=[K(C^Y):K(C^{\overline{k}})]C^{\overline{k}}$. Now from \autoref{lem:constant} we have
	\[
	D\cdot_k C  =d_C(\phi^*D\cdot_{\overline{k}} C^{\overline{k}}).
\]
	By the projection formula we have \[\pi^*D\cdot_{\overline{k}}C^{Y}=\phi^*D\cdot_{\overline{k}}\psi_*(C^Y) = [K(C^Y):K(C^{\overline{k}})](\phi^*D\cdot_{\overline{k}}C^{\overline{k}}).\]
	
	Thus \[D\cdot_k C=\frac{d_C}{[K(C^Y):K(C^{\overline{k}})]} (\pi^*D\cdot_{\overline{k}}C^{Y})\leq d_C (\pi^*D\cdot_{\overline{k}}C^{Y}).\]
	\end{proof}

\subsection{Cone theorem for surfaces}

The aim of this subsection is to prove the following:

\begin{theorem}\label{thm:surface_cone}
	Let $X$ be a normal projective surface over an arbitrary field $k$, and let $B$ be an effective $\bR$-divisor such that $K_X+B$ is $\bR$-Cartier. Then there is a countable collection of curves $\{C_i\}_{i\in I}$ such that:
	\begin{enumerate}
		\item \[\overline{NE}(X)=\overline{NE}(X)_{K_X+B\geq 0}+\sum_{i\in I}\mathbb{R}_{\geq 0}[C_i]. \]
		\item For any ample $\mbR$-divisor $A$, there is a finite subset $I_A\subset I$ such that
		\[\overline{NE}(X)=\overline{NE}(X)_{K_X+B+A\geq 0}+\sum_{i\in I_A}\mathbb{R}_{\geq 0}[C_i]. \]
		\item The rays $\{\mbR\cdot[C_i]\}$ do not accumulate in $\overline{NE}(X)_{(K_X+B)<0}$.
		\item  For each $C_i$, either:
		\begin{enumerate}
			\item $C_i$ is contained in $\Supp(B)$, or
			\item There is a unique positive integer $d_{C_i}$ depending only on $X, C$ and the ground field $k$  such that 
			\[ 0<-(K_X+B)\cdot_k C_i\leq  4d_{C_i}\]
		and if $L$ is any Cartier divisor on $X$, then $L\cdot_k C_i$ is divisible by $d_{C_i}$.
		\end{enumerate}
	\end{enumerate}
	Furthermore, if $(X,B)$ is a log canonical pair then every curve can be taken from the case (b).
\end{theorem}

Our starting point is \cite[Theorem 7.5]{Tan15f}, which  proves (2), that is to say that for any given ample $\mbR$-divisor $A$, there are finitely many curves $C_i$ such that 
\[\overline{NE}(X)=\overline{NE}(X)_{K_X+B+A\geq 0}+\sum_{i\in I_A}^n\mathbb{R}_{\geq 0}[C_i].\]

\begin{lemma}\label{lem:imperfect-curve} \cite[Lemma 3.2]{BCZ18}
	Let $C$ be a projective curve over an arbitrary field $k$, with local complete intersection singularities. Assume that $\deg_k K_C<0$ and $\ell=H^0(C, \mcO_C)$. Then
	\begin{enumerate}
		\item \label{itm:pic_trivial} $\Pic^0(C):=\{\msL\in\Pic(C)\; |\; \deg_k\msL=0\}=\{\mcO_C\}$,
		\item $C$ is a conic over $\ell$ embedded into $\mbP^2_\ell$, and $\deg_\ell K_C=-2$,
		\item if $\chr(\ell)>2$ and $C$ is normal, then $C_{\bar{\ell}}=C\otimes_\ell\bar{\ell}\cong\mbP^1_{\bar{\ell}}$,
	\end{enumerate}	
\end{lemma}

\begin{proposition}\label{prop:curve_hunt}
	Let $X$ be a normal projective surface defined over an arbitrary field $k$, and let $B$ be an effective $\bR$-divisor such that $K_X+B$ is $\bR$-Cartier. Suppose that $L$ is a nef Cartier divisor such that $L^{\perp}$ cuts out a $(K_X+B)$-negative extremal ray $R$.

	Then $R$ contains a curve $C$ such that either 
	\begin{enumerate}
		\item $C$ is contained in $\Supp(B_{>1})$, or
		\item $(X, B)$ is log canonical at the generic point of $C$, and there is an integer $d_{C}$  such that $L\cdot_k C$ is divisible by $d_{C}$ for every Cartier divisor $L$, and
		\[0<-(K_X+B)\cdot_k C\leq  4d_{C}.\]
	\end{enumerate}

\end{proposition}	
\begin{proof}
	\emph{Step 1}: We claim that we may assume $X$ is regular.\\
	Let $\pi:Y\to X$ be a $\bQ$-factorial dlt model of $(X,B)$, that is a model such that $(Y,(B_Y)_{\leq1})$ is dlt and the exceptional locus of $\pi$ lies over the non-klt locus of $(X,B)$.  Furthermore, assume that $Y$ is regular via replacing it with the minimal resolution, and set $K_Y+B_Y=\pi^*(K_X+B)$.  The construction ensures that $B_Y\geq 0$ and $(Y,(B_Y)_{\leq 1})$ is dlt.  Then $(\pi^*L)^\perp$ cuts out some face of $\overline{NE}(Y)$ which contains a $(K_Y+B_Y)$-negative class of $\overline{NE}(Y)$ by the projection formula.  This class is also $(K_Y+B_Y+A)$-negative for some sufficiently small ample $\mathbb{R}$-divisor $A$. Therefore by \cite[Theorem 7.5]{Tan15f}, $(\pi^*L)^{\perp}$ contains some $(K_Y+B_Y+A)$-negative extremal ray $R_Y$ which is generated by an integral curve $C_Y$. In particular, $C_Y$ is not contracted by $\pi$ because by the ampleness of $A$ we have $(K_Y+B_Y)\cdot C_Y<0$, and thus $C_Y$ is birational to its image, say $C=\pi_*C_Y$.  If $L$ is an $\bR$-Cartier divisor on $X$, by the projection formula we have 
	\[\pi^*L\cdot_k C_Y=L\cdot_k \pi_*C_Y= L\cdot_k C\]
	Consequently, we obtain $(2)$ if we can prove the existence of the constant $d$ and the bound for the curve $C_Y$ and $K_Y+B_Y=\pi^*(K_X+B)$.
	  Furthermore, if $C_Y$ is contained in $\Supp(B_{Y>1})$ and is not contracted over $X$, then  $C$ is contained in $\Supp(B_{>1})$.  

	Therefore if we can find $C_Y$ in $R_Y$  satisfying (1) or (2) for $(K_Y+B_Y)$, then we are done.
	
	\medskip

	\emph{Step 2}: Reduction to the log canonical case.
	
	By removing the components of $B_Y=\sum b_i B_i$ with $B_i\cdot R_Y\geq0$, we may assume that all irreducible components $B_i$ satisfy $B_i\cdot R_Y<0$.  If there is such a component $B_i$, then there can be at most one, and its support is equal to $C$, and if the coefficient of $B_i$ were greater than $1$, we are in Case (1) and are done.  So from now on we assume that all coefficients of $B_Y$ are at most $1$, in which case $(Y, B_Y)$ is log canonical (in fact dlt). 
	
	\medskip
	
	We have reduced the proof of \autoref{prop:curve_hunt} to the case where $(X,B)$ is dlt and $X$ is regular, by replacing  $X$ by $Y$.
	By \cite[Theorem 4.4]{Tan18}, there is a projective contraction $f:X\to Z$ contracting the ray $R$.  
	
	\medskip

	\emph{Step 3}:  The case of $\dim(Z)<\dim(X)$.\\
	In this case either $X$ has Picard number $\rho(X)=1$ or $X$ is a fibration over a curve. In either case we have $B_i\cdot R\>0$ for each $i$ because the curves in $R$ are movable, and so $K_X\cdot C<0$; in particular, in this case we may assume that $B=0$. We will consider two cases depending on the dimension of $Z$.\\ 
	
	\noindent
	\emph{Case I:} Assume that $\dim(Z)=1$.   Let $Y$ be the normalization of $(X\otimes_k\overline{k})_{\red}$ and $W$ the normalization of $Z\otimes_k \overline{k}$. 
	\begin{equation}
	\xymatrixrowsep{3pc}\xymatrixcolsep{3pc}\xymatrix{
		Y\ar[r]\ar[d]_g\ar@/^2pc/[rr]^-{\phi} & X\otimes_k\overline{k}\ar[r]\ar[d] & X\ar[d]^f\\
		W\ar[r] & Z\otimes_k\overline{k}\ar[r] & Z
	}
	\end{equation}
	Then by \cite[Theorem 4.2]{Tan15f} there is an effective Weil divisor $\Delta$ on $Y$ such that $\phi^*K_X=K_Y+\Delta$. Thus $-(K_Y+\Delta)$ is ample over $W$.  Replacing $W$ by the Stein factorization of $g:Y\to W$, we may assume that $g_*\mcO_Y=\mcO_W$. Now since $Y$ is a disjoint union of irreducible components and $g$ has connected fibers, it follows that $W$ is also a disjoint union of irreducible components. Furthermore, from the connectedness of $g$ it also follows that for an irreducible component $W'\subseteq W$, $Y'=g^{-1}W'$ is an irreducible component of $Y$. In particular, replacing $Y$ by $Y'$ and $W$ by $W'$, we may assume that $Y$ is an irreducible normal surface, $W$ an irreducible normal curve over the field $\overline{k}$, $g_*\mcO_Y=\mcO_W$ and $-(K_Y+\Delta)$ is $g$-ample. Then from \cite[Corollary 7.3]{Bad01} it follows that the general fibers of $g$ are integral curves.  Let $C_Y$ be such a general fiber. Then $$K_Y\cdot_{\overline{k}}C_Y\leq (K_Y+\Delta)\cdot_{\overline{k}}C_Y<0.$$ But also  $$\deg_{\overline{k}}K_{C_Y}=(K_Y+C_Y)\cdot_{\overline{k}}C_Y=K_Y\cdot_{\overline{k}}C_Y<0.$$ Since $C_Y$ is a local complete intersection curve (being a general fiber), by \autoref{lem:imperfect-curve} $C_Y$ is a conic in $\mbP^2_{\overline{k}}$ and $\deg_{\overline{k}}K_{C_Y}=-2$, i.e.,
	\[
	K_Y\cdot_{\overline{k}} C_Y=-2.
	\]
	
	Then by \autoref{lem:inequality}, there is a curve $C$ supported on a fiber  of $f:X\to Z$ 
	such that
	\[
	-(K_X\cdot_k C)\leq -d_C(K_Y+\Delta)\cdot_{\overline{k}}C_Y\leq 2d_C,
	\]
	where $d_C$ is the constant which appears in \autoref{lem:integer}.\\

	\noindent
	\emph{Case II:}
	Assume that $Z=\Spec(k)$. In this case  $X$ is a regular del Pezzo surface with $\rho(X)=1$ and $H^0(X, \mcO_X)=k$.   Let $Y$ be the normalization of $(X\otimes_k\overline{k})_{\red}$ and $\pi:Y\to X$ the induced morphism. Then by \cite[Theorem 4.2]{Tan15f} there exists an effective Weil divisor $\Delta$ on $Y$ such that $K_Y+\Delta\sim \phi^*K_X$.   Since $\rho(X)=1$, any curve on $X$ lies in the required extremal ray, so by \autoref{lem:inequality} it is sufficient to find any curve on $Y$ such that $-4\leq (K_Y+\Delta)\cdot_{\overline{k}} C_Y<0$.
	
	By cutting down by general hyperplanes we can find a curve $D$ on $Y$ such that $Y$ is smooth near $D$, $D$ is not contained in $\Supp(\Delta)$ and $-(K_Y+\Delta)\cdot_{\overline{k}} D>0$, as $-(K_Y+\Delta)$ is nef and big.  Fix a point $x\in D$ such that $x\notin\Supp(\Delta)$.  Then we can apply \cite[Theorem II.5.8]{kollar_rational} to obtain a rational curve $C_Y$ on $Y$ which passes through $x$ and which satisfies
	$$0<-(K_Y+\Delta)\cdot_{\overline{k}} C_Y\leq 4\frac{-(K_Y+\Delta)\cdot_{\overline{k}} C_Y}{-K_Y\cdot_{\overline{k}} C_Y}\leq 4.$$

	\bigskip
	
	\emph{Step 4:} Finding the curve $C$ in the case where $f$ is birational.
	
	As $X$ is a surface, $C$ is the unique curve contracted by $f$ and it satisfies $C^2<0$.  Let $\lambda_0$ be such that  $B=B'+\lambda_0C$ with $C\notin \Supp(B')$.   We wish to reduce to the case of  $\lambda_0=1$.  First note that we may remove components of $B$ which are not $C$, so assume $B'=0$.
	
	Let $\lambda_1=\lct(X; C)$, so that $(X, \lambda_1 C)$ is lc but not klt.  As $X$ is regular, we have that $\lambda_1>0$, and we also have $\lambda_1\geq\lambda_0$.  Let $\pi:X_1\to X$ be  a dlt model of $(X, \lambda_1 C)$; note that such a model exists via standard application of MMP as in \cite{Tan18}.  Then $(X_1,B_{X_1}+\lambda_1C_{X_1})$ is dlt where $C_{X_1}$ is the strict transform of $C$ and $K_{X_1}+B_{X_1}+\lambda_1C_{X_1}=\pi^*(K_X+\lambda_1C)$. Then we have 
		\begin{align*}
		-(K_{X_1}+B_{X_1}+\lambda_1 C_{X_1})\cdot_k C_{X_1} &=-(K_X+\lambda_1C)\cdot_k\pi_*C_{X_1}\\
		& = -(K_X+\lambda_1C)\cdot_k C\\
		& \geq -(K_X+\lambda_0C)\cdot_k C\\
		& \geq -(K_X+B)\cdot_k C>0
		\end{align*}
	
  Therefore we will be done if we can show that the required constant and bound holds for \[(K_{X_1}+B_{X_1}+\lambda_1C_{X_1})\cdot C_{X_1}\geq -4d_{C_{X_1}}.\]  
	
	Furthermore, we have that $C_{X_1}^2<0$ via the projection formula:
	$$0>C^2=C\cdot\pi_*C_{X_1}=\pi^*C\cdot C_{X_1}=(C_{X_1}+E)\cdot C_{X_1}\geq C_{X_1}^2$$
	for some effective $\pi_1$-exceptional divisor $E$.
	
	As before, we can remove the components of $B_{X_1}$ which are not $C$, and let $\lambda_2=\lct(X_1; C_{X_1})$.  Also, notice that $\lambda_2\geq \lambda_1$ and equality holds if and only if $\lambda_1=1$. In particular, repeating the above argument we create an increasing sequence  of $\{\lambda_i\}$, i.e., $\lambda_{i+1}\>\lambda_i$ for all $i\>1$ where equality holds if and only $\lambda_i=1$.    Thus by the ACC for log canonical thresholds as in \cite[Theorem 6.1]{Das18}, the sequence $\{\lambda_i\}$ must stabilize after finitely many stages, i.e., there exists an $n\in\mbN$ such that $\lambda_i=1$ for all $i\>n$. Then replacing $(X, B'+\lambda_0C)$ by $(X_n, \lambda_nC_{X_n})$ we may assume that $(X, C)$ is log canonical. Then again replacing $(X, C)$ by a dlt model and removing other components we may assume that $(X, C)$ is  dlt. Consequently, $C$ is normal, since it is a dlt center (see \cite[Lemma 3.4]{BCZ18}).  Then if $d_C$ is the constant from \autoref{lem:integer}, we have:
	
		\begin{align*}
		0<-(K_X+C)\cdot_k C=\deg_k-(K_C+B'_C) 
		&<\deg_k -K_C\\
		&=[k_C:k]\deg_{k_C}-K_C\\
		&=2[k_C:k]\quad ({\rm{by\ \autoref{lem:imperfect-curve}}})\\
		&\leq 2 d_C 
		\end{align*}

where $[k_C:k]|d_C$ by the first reduction in the proof of \autoref{lem:integer}.

\end{proof}

\begin{lemma}\label{lem:surface_rationality_theorem}
	Let $(X,B\>0)$ be a log canonical pair of dimension $2$, where $B$ is a $\bQ$-divisor. Assume that $K_X+B$ is not nef. Then there is a natural number $n$ depending only on $(X,B)$ such that if $H$ is an ample Cartier divisor, and  \[\lambda=\inf\{t\>0:K_X+B+tH\mathrm{\ is\ nef}\}\]  then $\lambda=\frac{n}{m}$ for some natural number $m$. Moreover, there is a curve $C$ generating an extremal ray $R=\mbR_{\>0}\cdot[C]$ of $\overline{\NE}(X)$ such that there is an integer $d_{C}$  such that $L\cdot_k C$ is divisible by $d_{C}$ for every Cartier divisor $L$, and
	\[
	(K_X+B+\lambda H)\cdot C=0\quad\mbox{and}\quad -(K_X+B)\cdot C\<4d_C.
	\]
\end{lemma}	

\begin{proof}
	Note that since $(X, B)$ is log canonical, by \autoref{prop:curve_hunt} it is enough to prove that there is a $(K_X+B)$-negative extremal ray $R$ such that $(K_X+B+\lambda H)\cdot R=0$; indeed in that case there is an integral curve $C$ such that $R=\mbR_{\>0}\cdot[C]$ and $(K_X+B+\lambda H)\cdot C=0$. Solving this equation shows that $\lambda$ is a rational number. We also need to show that the numerator of $\lambda$ can be chosen uniformly. 
	
	To this end, first we claim that there exists an increasing sequence of real numbers $\lambda_i$ converging to $\lambda$ and $K_X+B$-negative extremal rays $R_i$ such that $(K_X+B+\lambda_i H)\cdot R_i=0$. Let $t_i>0$ be an increasing sequence of non-negative real numbers such that $t_i<\lambda$ for all $i$ and $\lim t_i=\lambda$. Then $K_X+B+t_iH$ is not nef for any $i$, and thus by \cite[Theorem 7.5]{Tan15f} there is a $(K_X+B+t_iH)$-negative extremal ray $R_i=\mbR_{\>0}\cdot[C_i]$ generated by a curve $C_i$ satisfying $0<-(K_X+B+t_iH)\cdot C_i\<4d_{C_i}$ (this follows from \autoref{prop:curve_hunt}). Choose $\ve_i>0$ such that  $(K_X+B+(t_i+\ve_i)H)\cdot C_i=0$ for all $i$.

	Now recall that $(K_X+B+\lambda H)$ is nef, in particular, $(K_X+B+\lambda H)\cdot C_i\>0$ for all $i$. This implies that $t_i\<t_i+\ve_i\<\lambda$ for all $i$, otherwise if $t_i<\lambda<t_i+\ve_i$ for some $i$, then from the definition of $\ve_i$ it follows that $(K_X+B+\lambda H)\cdot C_i<0$, a contradiction. Thus by setting $\lambda_i:=t_i+\ve_i$ we see that $(\lambda_i)$ is an increasing sequence of positive real numbers converging to $\lambda$.

	Next we claim that the sequence $(\lambda_i)$ stabilizes to its limit $\lambda$ after finitely many steps. Clearly if the claim holds then we have $(K_X+B+\lambda H)\cdot C_i=(K_X+B+\lambda_i H)\cdot C_i=0$ for all $i\gg0$, and this completes the proof. So to the contrary assume that $(\lambda_i)$ is strictly increasing.	
	As a result, for each ray $R_i$ we obtain a curve $C_i$ and constant $d_{C_i}$ as in \autoref{thm:surface_cone}.  In particular we have \[-4\leq \frac{(K_X+B)\cdot C_i}{d_{C_i}}\leq 0.\]  The property of $d_C$ ensures that the numbers \[\frac{I_{K_X+B}(K_X+B)\cdot C_i}{d_{C_i}}\] are integers, where $I_{K_X+B}$ is the Cartier index of $K_X+B$, so by the bound, there are only finitely many possibilities for these numbers independently of $R_i$.  Similarly, if we truncate the sequence so that $\lambda/2<\lambda_i$ for each $i$, there are only finitely many possibilities for \[\frac{(K_X+B+\frac{\lambda}{2}H)\cdot C_i}{d_{C_i}}\]
	This means finitely many possibilities for \[\frac{H\cdot C_i}{d_{C_i}}\] and so finitely many possible values of $t_i$, which is a contradiction.
	
	We have shown that the sequence $t_i$ did in fact stabilize at $\lambda$, and so the numerator of $\lambda$ divides the integer \[-\frac{I_{K_X+B}(K_X+B)\cdot C_i}{d_{C_i}}\] which in turn divides $4I_{K_X+B}!$.
\end{proof}

\begin{proof}[Proof of \autoref{thm:surface_cone}]
	Exactly as in \cite[Proof of 1.1]{BW17}, using \cite[Theorem 3.15]{KM98}.
\end{proof}

\subsection{Cone theorem for $3$-folds I}

\begin{proposition}\label{prop:3-fold_find_curve}
	Let $X$ be a normal $\mathbb{Q}$-factorial $3$-fold over a field of characteristic $p>0$, and $B$ be an  $\bR$-divisor with coefficients in $[0,1]$ such that $0\leq M\equiv K_X+B$ is $\bR$-Cartier. Let $R$ be a $(K_X+B)$-negative extremal ray. Then $R$ contains a curve $C$, which is either contained in \[\Sing(X)\cup\Sing(\Supp(M+B))\] or  there is an integer $d_{C}$  such that $L\cdot_k C$ is divisible by $d_{C}$ for every Cartier divisor $L$ and 
	\[0<-(K_X+B)\cdot C\leq 4 d_{C}.\]
\end{proposition}

\begin{proof}
	By assumption we have $K_X+B\num M\geq 0$.  Therefore any extremal ray  with $(K_X+B)\cdot R<0$ is contained in the image of $\overline{NE}(S)\to \overline{NE}(X)$ for some component $S$ of $\Supp(M)$ with $S\cdot R<0$.    
	
	Let $a\>0$ be a real number such that $S$ has coefficient $1$ in $K_X+B+aM$.  Then by adjunction there is $B_{\tilde{S}}$ such that $K_{\tilde{S}}+B_{\tilde{S}}=(K_X+B+aM)|_{\tilde{S}}$, where $\tilde{S}$ is the normalization of $S$.  There is a $(K_{\tilde{S}}+B_{\tilde{S}})$-negative extremal ray $R_{\tilde{S}}$ whose image in $\overline{\NE}(X)$ is $R$.  The support of $B_{\tilde{S}}$ is contained in the pre-image of $\Sing(X)\cup\Sing(\Supp(M+B))$.    
	Then $R_{\tilde{S}}$ contains the required curve $C$ by \autoref{thm:surface_cone}, and \[(K_{\tilde{S}}+B_{\tilde{S}})\cdot_k C=(K_X+B+aM)\cdot_k C\leq (K_X+B)\cdot_k C<0.\]
\end{proof}

\begin{lemma}\label{lem:rationality_theorem}
	Let $(X,B)$ be a projective $\bQ$-factorial $3$-fold pair such that  $B$ is a $\bQ$-divisor with coefficients in $[0, 1]$,  $\kappa(K_X+B)\geq0$ and $K_X+B$ is not nef.  Then there is a natural number $n$ such that for any ample Cartier divisor $H$, if  \[\lambda=\inf\{t:K_X+B+tH\mathrm{\ is\ nef}\}\]  then $\lambda=\frac{n}{m}$ for some natural number $m$.
\end{lemma}	
\begin{proof}
	It follows exactly as in \autoref{lem:surface_rationality_theorem}, using \autoref{prop:3-fold_find_curve} and \cite[Theorem 7.6]{Tan15f}, except that finitely many of the extremal rays involved may contain curves from \[\Sing(X)\cup\Sing(\Supp(M+B))\] which do not satisfy the usual length bound.  These do not cause issues as there are only finitely many such curves, but may result in a larger value of $n$.
\end{proof}
\begin{theorem}\label{thm:cone_theorem}
	Let $X$ be a normal $\mathbb{Q}$-factorial $3$-fold over an arbitrary field of characteristic $p>0$. Let $B$ be an $\bR$-divisor with coefficients in $[0,1]$ such that $K_X+B$ is $\mbR$-Cartier. If $\kappa(K_X+B)\geq 0$, then there is a countable collection of curves $\{C_i\}_{i\in I}$ such that:
	\begin{enumerate}
		\item \[\overline{NE}(X)=\overline{NE}(X)_{K_X+B\geq 0}+\sum_i\mbR_{\>0}\cdot[C_i].\]
		\item The rays $[C_i]$ do not accumulate in the half space $(K_X+B)_{<0}$.
		\item  All but finitely many $C_i$ satisfy
		\[0<-(K_X+B)\cdot C_i\leq 4 d_{C_i}.\]
	\end{enumerate}
\end{theorem}

\begin{proof}
	By \cite[Theorem 3.15]{KM98} it is enough to show that there is some integer $a(K_X+B)$ such that if $N$ is a nef Cartier divisor and if \[r:= \max\{t\in \bR: H+t(K_X+B) \mathrm{\ is\ nef}\},\] then $r$ is a rational number of the form $n/a(K_X+B)$ for some integer $n$.  This follows immediately from \autoref{lem:rationality_theorem}.
\end{proof}

The idea behind this bound on the length of extremal rays is that in many applications, one can replace a curve $C$ by the $1$-cycle $\frac{1}{d_{C}}C$ and run the exact same proofs as in the algebraically closed case.  We give the following as an example, which is a standard application of the length of extremal rays.

\begin{proposition}\cite[3.8]{BW17}\label{lem:length_nef_perturb}
	Let $X$ be a $\bQ$-factorial projective klt surface over $k$.  Let $V$ be a finite dimensional rational affine space of Weil divisors, and let \[\sL=\{0\<\Delta\in V:(X,\Delta) \mathrm{\ is\ lc}\}\]
	Fix $B\in \sL$.  Then there are real numbers $\alpha,\delta>0$ depending on $(X,B)$ and $V$ such that
	\begin{enumerate}
		\item if $\Gamma$ is any extremal curve and $(K_X+B)\cdot\frac{\Gamma}{d_{\Gamma}}>0$ then $(K_X+B)\cdot\frac{\Gamma}{d_{\Gamma}}>\alpha$.
		\item  If $\Delta\in \sL$, $||\Delta-B||<\delta$ and $(K_X+\Delta)\cdot R\leq 0$ for an extremal ray $R$ then \[(K_X+B)\cdot R\leq 0.\]
	\end{enumerate}
\end{proposition}
\begin{proof}
	The proof of the proposition follows from that of \cite[3.8]{BW17} word for word after substituting $\frac{1}{d_{\Gamma}}\Gamma$ for $\Gamma$.  We include only that of the first part.
	\begin{enumerate}
		\item If $B$ is a $\bQ$-divisor then this is obvious.  Otherwise let $B=\sum a_iB_i$ with $a_i\geq 0$, $\bQ$ divisors $B_i>0$ and  $\sum a_i=1$.  
		We have \[(K_X+B)\cdot\frac{\Gamma}{{d_\Gamma}}=\sum a_i(K_X+B_i)\cdot\frac{\Gamma}{d_\Gamma}\]
		Assuming $(K_X+B)\cdot \frac{\Gamma}{{d_\Gamma}}<1$ then  there are only finitely many possibilities for the intersection numbers $(K_X+B_i)\cdot \frac{\Gamma}{d_{\Gamma}}$ because $(K_X+B_i)\cdot\frac{\Gamma}{d_\Gamma}\geq -4$.  This in turn implies that there are only finitely many possibilities for $(K_X+B)\cdot\frac{\Gamma}{d_{\Gamma}}$.  So the existence of $\alpha$ is clear.
		\item See \cite[Proposition 3.8]{BW17} and proceed as above.
	\end{enumerate}
\end{proof}

\begin{remark}
	The same statement for $3$-folds over $F$-finite fields follows from the improved \autoref{thm:cone} which we prove later.
	\end{remark}

\subsection{Base point free theorem I}

The following theorem was claimed over arbitrary fields in \cite[Theorem 0.5]{Kee99}.  However, the proof applied a base change of the ground field to reduce to the case of an algebraically closed field, and the hypotheses of the theorem are not stable under this base change (see  \cite[Theorem 1.1]{Tan15f}, \cite[Theorem 1.1]{PW17} and \cite[Theorem 1.1]{ji_waldron}).  In this section we prove this theorem over arbitrary fields by following Keel's original arguments along with the canonical bundle formula for purely inseparable field extensions developed in the above mentioned works.

\begin{theorem}[Keel's Base-Point Free Theorem]\label{thm:keel's-bpf-thm}
	Let $X$ be a normal projective $3$-fold over an arbitrary field $k$ of positive characteristic.  Let $L$ be a nef and big Cartier divisor on $X$.  If $L-(K_X+\Delta)$ is nef and big for some boundary divisor $\Delta$ with coefficients in $[0,1)$, then $L$ is EWM.  
\end{theorem}

\begin{remark}
	Note that Keel's theorem \cite[Theorem 0.5]{Kee99} had a `$\mbQ$-factoriality' assumption on $X$, however, it was never used in his proof. So we remove this extra condition on $X$ here. 
\end{remark}

\begin{proof}
	First replacing $k$ by a finite extension we may assume that $H^0(X, \mcO_X)=k$. Let $k^s$ be the separable closure of $k$, and $X_{k^s}=X\otimes_k k^s$ the base change. Since $X$ is normal and $k^s/k$ is separable, $X_{k^s}$ is normal. Note that $X_{k^s}$ descends to a finite sub-extension $k'/k\subset k^s/k$, in particular, $\psi:X_{k'}\to X$ is a finite \'etale morphism and thus $K_{X_{k'}}=\psi^*K_X$. Consequently, if $\phi:X_{k^s}\to X$ is the projection, then $K_{X_{k^s}}=\vphi^*K_X$. Thus replacing $k$ by $k^s$ and using \cite[Lemma 2.2]{Kee99} from now on we may assume that $k$ is separably closed and $H^0(X, \mcO_X)=k$.
	In the following steps we will closely follow the proof and notations of \cite[Theorem 0.5]{Kee99}.\\ 
	
	\emph{Step 0:}
	Following the notation of the proof of \cite[Theorem 0.5]{Kee99}, write $L=A+N_0+N_1+N_2$, where $A$ is ample, $N_i\geq 0$ and the restriction of $L$ to each component of $N_i$ has numerical dimension $i$.  By the same argument as in \cite[Theorem 0.5]{Kee99}, it is enough to show that $L|_T$ is EWM, where $T=\Supp(N_1)$.  Let $T_i$ be the irreducible components of $N_1$, and $\pi_i:\tilde{T_i}\to T_i$ be the corresponding normalizations\\   
	
	\emph{Step 1:}
	Fix $i$, and in the next several steps we will show that $L|_{T_i}$ is EWM. Consider the commutative diagram
	\[\xymatrix{
		\sC\ar[r]\ar[d] & \tilde{T_i}\ar[d] \\
		\sD\ar[r] & T_i
	}\]
	defined by the conductor schemes of $\pi:\tilde{T_i}\to T_i$.

	The morphism $\tilde{T}_i\to T_i$ is an isomorphism outside $\Supp(\sC)$.  Let $S_i$ be the normalization of $(\tilde{T}_i\otimes_k\overline{k})_{\mathrm{red}}$. Then we have the following digram:
	
		\[	\xymatrix{
		S_i\ar[r] & \tilde{T_i}\otimes_k\overline{k}\ar[d]\ar[r]& \tilde{T}_i\ar[d] \\
		& T_i\otimes_k\overline{k}\ar[r]& T_i
	}
	\]
	
Note some important facts about this diagram:
\begin{enumerate}
	\item It follows from \cite[Theorem 1.1(b)]{ji_waldron} that there are effective Weil divisors $\frak F_i$ and $\frak M_i$  on $S_i$ such that the $\Supp (\frak{F}_i)$ is equal to the locus where $S_i\to (\tilde{T_i}\otimes_k \overline{k})_{\red}$ fails to be an isomorphism, and satisfies
	\begin{equation}\label{eqn:inseparable-cbf}
	K_{S_i}+\frak F_i+\frak M_i\sim \sigma^*K_{\tilde{T}_i}
	\end{equation}
	Though \cite[Theorem 1.1]{ji_waldron} assumes that the ground field of $\tilde{T}_i$ is a function field, we may find such subfields of $\overline{k}$ over which  $S_i$ and $T_i$ are defined, via the process outlined in \cite[Subsection 2.1]{ji_waldron}.  After obtaining the divisor $\frak F_i$ on those subfields, we can tensor back up to obtain it on $S_i$ and $T_i$.
	\item $\tilde{T_i}\otimes_k\overline{k}\to T_i\otimes_k\overline{k}$ is an isomorphism outside of $\Supp(	\sC\otimes_k\overline{k})$, and hence so is $(\tilde{T_i}\otimes_k\overline{k})_{\red}\to (T_i\otimes_k\overline{k})_{\red}$
	\item $S_i$ is the normalization of $(T_i\otimes_k \overline{k})_{\red}$  by \cite[Prop. 1.22, Ch. 4]{Liu02}, since $S_i$ is normal and $S_i\to (T_i\otimes_k \overline{k})_{\red}$ is a finite birational morphism.
	\item $S_i\to (T_i\otimes_k\overline{k})_{\red}$ is an isomorphism outside of $\Supp(\sC)\cup \Supp(\frak F_i)$.\\
	\end{enumerate}

Now similar to the \cite[Eqn (5.0.3)]{Kee99} we have an effective Weil divisor $D_i$, an effective $\bQ$-divisor $R_i$ and an ample $\bQ$-divisor $A_i$ on $S_i$ such that
\begin{equation}\label{eqn:base-change-adjunction}
(1+\lambda_i)L|_{S_i}=K_{\tilde{T}_i}+D_i+R_i+A_i,	
\end{equation} 
where $\Supp(D_i)$ contains the support of $\frak F_i$ and the the pullback of the support of $\mcC$. As $S_i$ is defined over an algebraically closed field, using the Riemann-Roch argument from  Keel's proof on  the divisor $D_i$ in place of his $Q_i$ we see that $L|_{S_i}$ is EWM.  

		By \autoref{eqn:base-change-adjunction}, the general fibers of the associated map $h_i:S_i\to Z_{S_i}$ are conics, and so $D_i$ has at most one horizontal component, which is of degree one to $Z_{S_i}$ and has coefficient one.  So in particular, $\frak F_i+C_i$ intersects the general fibers of $h_i$ in at most $1$ point, and thus $L|_{({T}_i\otimes_k\overline{k})_{\red}}$ is EWM by \cite[Corollary 2.15]{Kee99}. Then from \cite[Lemma 1.5]{Kee99} it follows that $L|_{{T}_i\otimes_k \overline{k}}$ is EWM, and hence from \cite[Lemma 2.2]{Kee99} that $L|_{T_i}$ is EWM.

	\emph{Step 3:}  Let $g_i:T_i\to Z_i$ be the associated morphism for each $i$. We need to glue these $g_i$'s to show that $L|_T$ is EWM, where $T=\cup_i T_i$. The rest of Keel's argument works essentially unchanged.  By induction on $i$, it is enough to show that we can glue $g_W$ to $f_i$ where $W'=\cup_{i=1}^{n-1}T_i$ where we assume that $L|_{W'}$ is EWM.    But the divisors $T_i$ in which $T_i$ meets $W$ are all contained in the support of $D_i$, and we have already seen that the map $g_i|_{Q_i}$ has geometrically connected fibers at all but finitely many points.  We conclude by \cite[Corollary 2.12]{Kee99}.
	
\end{proof}

\subsection{Contraction of extremal rays}
\begin{theorem}[Contraction Theorem]\label{thm:contraction_theorem}
	Let $(X, \Delta)$ be a normal $\mbQ$-factorial $3$-fold dlt pair, projective over an arbitrary field of characteristic $p>0$, such that $K_X+\Delta$ is pseudo-effective. Let $R$ be a $(K_X+\Delta)$-negative extremal ray. Then the corresponding contraction $f:X\to Z$ exists in the category of algebraic spaces. Moreover, if $(X, \Delta=S+B)$ is plt, $S$ is normal and $S\cdot R<0$, then $f:X\to Z$ is a birational morphism to a projective $\mbQ$-factorial $3$-fold with $\rho(X/Z)=1$.
\end{theorem}

\begin{proof}
	Follows similarly as in the proof of \cite[Theorem 5.3]{HX15} using Cone \autoref{thm:cone_theorem}, \autoref{lem:rationality_theorem} and Keel's base-point free \autoref{thm:keel's-bpf-thm}. For the plt case we use \cite[Theorem 1.1]{Tan15} instead of \cite[Lemma 2.3.5]{KK}.
\end{proof}

\subsection{LMMP with scaling and weak Zariski decompositions}

We can use the modified cone theorem to find the required extremal rays for these LMMPs as defined in \cite{Bir16}, and run those LMMPs assuming the existence of the required contractions and flips.  The proofs of \cite[Sec 3]{Bir16} go through by changing all expressions of the form $D\cdot \Gamma$ for some divisor $D$ into $D\cdot\frac{\Gamma}{d_{\Gamma}}$ and using \autoref{thm:surface_cone} in place of \cite{Tan12j}.

\section{Complements on surfaces over $F$-finite fields}\label{sec:complements}

In this section we will prove the following theorem.
\begin{theorem}\label{thm:globally-f-regular}
	Let $k$ be an $F$-finite field of characteristic $p>5$. Let $(S, B)$ be a $2$-dimensional pair over $k$ with a proper birational morphism $f:S\to T$ to a normal surface germ $(T, 0)$ such that 
	\begin{enumerate}
		\item $(S, B)$ is klt,
		\item $-(K_S+B)$ is $f$-nef, and
		\item the coefficients of $B$ are in the standard set $\left\{1-\frac{1}{n}:n\in\mbN\right\}$.
	\end{enumerate}
	Then $(S, B)$ is globally $F$-regular over $T$.\\ 
\end{theorem}

Note that this is a generalization of \cite[Theorem 3.1]{HX15} to a larger class of ground fields.  This is in turn a generalization of the  fact due to Hara \cite[Cor 4.9]{Har98} that over an algebraically closed field $k$ of $\chr p>5$, a $2$-dimensional pair $(S, 0)$ over $k$ is strongly $F$-regular if and only if $(S, 0)$ is klt.  That result was proved in our situation by  Sato and Takagi \cite[Theorem 1.2]{ST18}.

First we need some results on complements. For definition and basic properties of complements see \cite{Pro01}.  First we need an analogue of the classification of curve complements as in \cite{Pro01}.

\begin{lemma}\label{lem:complement-on-curve}
	Let $C$ be a regular curve over a field $k$ such that $\deg_k K_C<0$. Let $(C, \Delta)$ be a klt pair such that $-(K_C+\Delta)$ is nef and the coefficients of $\Delta$ are contained in the standard set $\mcS$ (see \autoref{def:standard-set}). Then $(C,\Delta)$ is either $1$-, $2$-, $3$-, $4$- or $6$-complementary.
	\end{lemma}
\begin{proof}
	Note that the conclusion is unaffected by the choice of the base field, so extending the base field $k$ if necessary we may assume that $k=H^0(C,\sO_C)$.  As $C$ is a regular curve, a pair $(C,\Delta^+)$ on $C$ is log canonical if and only if the coefficients of $\Delta^+$ are at most one.  Furthermore, if $\Delta^+$ is a $\bQ$-boundary, then $K_C+\Delta^+\sim_{\bQ}0$ if and only if $\deg_k \Delta^+=2$ by \autoref{lem:imperfect-curve}\autoref{itm:pic_trivial}.  It also follows from \autoref{lem:imperfect-curve}\autoref{itm:pic_trivial} that if $m$ is a common denominator of the coefficients of $\Delta^+$, then $m(K_C+\Delta^+)\sim 0$.  So the result follows so long as we can find a divisor $\Delta^+\geq \Delta$ such that the coefficients of $\Delta^+$ have denominator $1$, $2$, $3$, $4$ or $6$ and $\deg_k\Delta^+=2$.  Note that $\deg_k \Delta=\sum_{P\in\Supp \Delta} [k(P):k]a_P$.  So we are reduced to showing that if we have a set of positive integers $n_i=[k(P_i):k]$ and standard coefficients $a_i$ such that $\sum_i n_ia_i\leq 2$ then we can find rational numbers $a_i^+\geq a_i$ with denominators as described above satisfying $\sum_i n_i a_i^+=2$.
	This is proved in \autoref{lem:numerical_complement}.
	\end{proof}
Now we define some notations for the solutions of $\sum_i n_ia_i\<2$, where $n_i\in\mbZ^+$ and $x_i\in\mcS$.  We will denote such a solution by
$(a_1, a_2, a_3, a_4)_{(n_1, n_2, n_3, n_4)}$. If $n_i=1$ for all $i$, then we will simply denote the solution by $(a_1, a_2, a_3, a_4)$. Note that since we are concerned about finding $\Delta$ with $\deg_k\Delta\<2$, we will make distinction among solutions of the form $(a, a, a, a)$, $(a, a)_{(2, 2)}$ and $(a, a)_{(1, 3)}$, since they correspond to different divisors on $C$. Explicitly, a solution of the form $(a_1, a_2, a_3, a_4)_{(n_1, n_2, n_3, n_4)}$ corresponds to a divisor $\sum_{i=1}^4 a_iP_i$ on $C$ such that $[k(P_i):k]=n_i$ for $i=1, 2, 3, 4$.

\begin{lemma}\label{lem:numerical_complement}
	Let $a_i\in\mcS-\{0, 1\}$, and $n_i$ be positive integers such that $\sum_i n_ia_i\leq 2$, where $\mcS$ is the standard set.  Then there are rational numbers $a_i^+\in\mbQ^{+}$ satisfying: $a_i\<a_i^+\<1$ for all $i$ and $\sum_i n_i a_i^+=2$. Moreover, the denominators of the $a_i^+$'s can be simultaneously cleared by multiplying by some $N\in \{1,2,3,4,6\}$.
	\end{lemma}
\begin{proof}
	Since $a_i\geq 1/2$, it follows that $\sum n_i\leq 4$. Also, since the $a_i$ satisfy the DCC, there are only finitely many solutions $a_i\in\mcS$ satisfying $\sum n_ia_i\<2$.  In \autoref{table:comp}, we enumerate these possibilities, first ordered by decreasing values of $\sum_i n_i$, and then by lexicographic order on the $a_i$.  Note that the first column represent the divisors $\Delta$ on $C$ satisfying the hypothesis of \autoref{lem:complement-on-curve}, the second column gives the corresponding complement $\Delta^+$ , and the third column gives the value of $N$ for which it is an $N$-complement.

\begin{table}[h!]
	
	\begin{tabular}{|c|c|c|c|  }
		\hline	
	 $(a_i)_{n_i}=\Delta$ &  $(a_i^+)_{(n_i)}=\Delta^+$ & $N$\\
	\hline
	\hline
	$(1/2,1/2,1/2,1/2)$ & $(1/2,1/2,1/2,1/2)$ & $2$\\
	\hline
	$(1/2,1/2,1/2)_{(1,1,2)}$ & $(1/2,1/2,1/2)_{(1,1,2)}$ & $2$\\
	\hline
	$(1/2, 1/2)_{2,2}$ & $(1/2, 1/2)_{2,2}$ & $2$ \\
	\hline
	$(1/2)_{(4)}$ & $(1/2)_{(4)}$  & $2$\\
	\hline
	\hline
	$(1/2,1/2,1-1/m)$  & $(1/2,1/2,1)$ & $2$\\
	\hline
	$(1/2, 1-1/m)_{(2,1)}$ & $(1/2,1)_{(2,1)}$ & $2$\\
	\hline
	$(1/2)_{(3)}$ & $(2/3)_{(3)}$ & $3$\\
	\hline
	\hline
	$(1/2,2/3,2/3)$& $(2/3, 2/3, 2/3)$& $3$\\
	\hline
	 $(1/2, 2/3)_{(1,2)}$ & $(2/3, 2/3)_{(1,2)}$& $3$\\
	\hline
	\hline
	$(1/2, 2/3, 3/4)$ & $(1/2,3/4,3/4)$ & $4$\\
	\hline
	$(1/2, 2/3, 4/5)$ & $(1/2, 2/3, 5/6)$ & $6$\\
	\hline
	$(1/2, 2/3, 5/6)$ & $(1/2, 2/3, 5/6)$ & $6$\\
	\hline
	\hline
	$(1/2,3/4,3/4)$ & $(1/2, 3/4, 3/4)$&$4$\\
	 \hline
	$(1/2, 3/4)_{(1,2)}$ & $(1/2,3/4)_{(1,2)}$ & $4$ \\
	\hline
	\hline
	$(2/3, 2/3, 2/3)$ & 	$(2/3, 2/3, 2/3)$& $3$\\
	 \hline
	$(2/3, 2/3)_{(1,2)}$ & $(2/3, 2/3)_{(1,2)}$ & $3$\\
	 \hline
	$(2/3)_{(3)}$ & $(2/3)_{(3)}$ & $3$ \\
	\hline
	\hline
	$(1-1/m_1,1-1/m_2)$ & $(1,1)$ & $1$ \\
	\hline
	$(1-1/m)_{(2)}$ &$(1)_{(2)}$ & $1$ \\
	\hline
	\hline
	$(1-1/m)$ & $(1)$ &$1$ \\
	\hline	
	\end{tabular}
	\caption{Complement}
	\label{table:comp}
\end{table}

	\end{proof}

As in \cite{HX15} we need the following theorem over $F$-finite fields in order to prove \autoref{thm:globally-f-regular}.
\begin{theorem}[{{c.f. \cite[Theorem 3.2]{HX15}}}]\label{thm:boundedness-of-complements}
	Let $k$ be a field of characteristic $p>0$. With the same notations and hypothesis as in \autoref{thm:globally-f-regular}, there exists a divisor $B^c\>B$ and an integer $N\in\mcR N_2=\{1, 2, 3, 4, 6\}$, such that $N(K_S+B^c)\sim_T 0$ and $(S, B^c)$ is log canonical but not klt. Let $\nu:\tilde{S}\to S$ be a dlt modification, $K_{\tilde{S}}+B^c_{\tilde{S}}=\nu^*(K_S+B^c)$. Then
	\begin{enumerate}
		\item If $(\tilde{S}, B^c_{\tilde{S}})$ is plt  and $C=\lrd B^c_{\tilde{S}}\rrd$ is irreducible, then we may assume that $(C, \Diff_C(B^c_{\tilde{S}}-C))$ belongs to one of the cases which appear in \autoref{table:comp}.
		\item If $(\tilde{S}, B^c_{\tilde{S}})$ is not plt and $C$ is an exceptional$/T$ curve contained in the $\Supp\lrd B^c_{\tilde{S}}\rrd$, then $N\in\{1, 2\}$ and $(C, \Diff_C(B^c_{\tilde{S}}-C))$ is of the type $(1/2, 1/2, 1)_{(1,1,1)}, (1, 1)_{(1,1)}, (1/2,1)_{(2, 1)}$ or $(1)_{(2)}$. \\
	\end{enumerate}
\end{theorem}

\begin{proof}[Proof of \autoref{thm:boundedness-of-complements}]
	It follows from the same arguments as in the proof of \cite[Theorem 3.2]{HX15} using \autoref{lem:numerical_complement} in place of 
	\cite[4.1.11, 4.1.12]{Pro01} to find curve complements. We also note that \cite[Pro. 4.4.1]{Pro01} is used in the proof of \cite[Theorem 3.2]{HX15} to lift complements on a birational model, fortunately the proof of \cite[Pro. 4.4.1]{Pro01} works for surfaces over imperfect fields since the relative Kawamata-Viehweg vanishing theorem is known for any birational morphism between excellent surfaces, see \cite[Theorem 10.4]{Kol13}.\\ 
	\end{proof}

\begin{proof}[Proof of \autoref{thm:globally-f-regular}]
	Using the same notations as in the proof of \cite[Theorem 3.1]{HX15} we show below that the proof of Hacon and Xu works over $F$-finite fields of characteristic $p>5$. The proof of \cite[Theorem 3.1]{HX15} involves several lemmas and propositions: Lemma 3.4, 3.5, 3.6, 3.7, 3.8, and Proposition: 3.9 and 3.10. In what follows we will show that all of these results hold over $F$-finite fields in $\chr p>5$.\\

	\begin{itemize}
		\item The proof of Lemma 3.4, 3.6, 3.7 and Proposition 3.9 works over $F$-finite fields without any change.
		\item For Lemma 3.5, 3.8 and Proposition 3.10 we either give a short proof explaining how to make the arguments of Hacon and Xu work over $F$-finite fields or reduce our problem to that of an algebraically closed base field.
	\end{itemize}

\end{proof}

\begin{lemma}\cite[Lemma 3.5]{HX15}\label{lem:hx-lemma-3.5}
With notations and hypothesis as in \cite[Lemma 3.5]{HX15}, it holds that $-(K_{\tilde{S}}+B^*_{\tilde{S}})$ is nef over $T$.	
\end{lemma}

\begin{proof}
	Since $B^*_{\tilde{S}}\<B^c_{\tilde{S}}$, we can write $-(K_{\tilde{S}}+B^*_{\tilde{S}})=-(K_{\tilde{S}}+B^c_{\tilde{S}})+E$, where $E$ is an effective $\mbQ$-divisor such that $E\wedge C=\emptyset$. Therefore we have $(K_{\tilde{S}}+B^*_{\tilde{S}})\cdot C\<0$. If $D$ is an exceptional curve not in $\Gamma^0_1$, then by the a similar argument it follows that $(K_{\tilde{S}}+B^*_{\tilde{S}})\cdot D\<0$.
	
	Now let $D=E_i$ be an exceptional curve contained in $\Gamma^0_1$. 
	 Let $B^c_{\tilde{S}}=\frac{p_{j-1}}{q} E_{j-1}+\frac{p_j}{q} D+\frac{p_{j+1}}{q} E_{j+1}+F$, where $E_{j-1}$ and $E_{j+1}$ are the  curves adjacent to $D$ in $\Gamma^0_1$, and $F$ supports the other components of $B^c_{\tilde{S}}$. Since $K_{\tilde{S}}+B^c_{\tilde{S}}\sim_{\mbQ, T} 0$, $(K_{\tilde{S}}+B^c_{\tilde{S}})\cdot D=0$. Let $\ell=H^0(D, \mcO_D)$, $\deg_{\ell} E_{j-1}|_D=n_{j-1}, \deg_\ell E_{j+1}|_D=n_{j+1}$ and $\deg_\ell F|_D=\frac{r}{q}$, where $n_{j-1}, n_{j+1}\in\mbZ_{\>1}$ and $r\in\mbZ_{\>0}$. We also note that the arithmetic genus $p_a(D/\ell)=0$, since $H^1(D, \mcO_D)=0$ by \cite[Lemma 10.8]{Kol13}. Thus $\deg_\ell (K_{\tilde{S}}+D)|_D=\deg_\ell K_D=2p_a(D/\ell)-2=-2$ (see \cite[Cor. 3.31, Chap. 7]{Liu02}).
	 Then from $\deg_\ell (K_{\tilde{S}}+B^c_{\tilde{S}})|_D=0$ we have
	 
	\begin{align}
		n_{j-1}\left(\frac{p_{j-1}}{q}\right)+n_{j+1}\left(\frac{p_{j+1}}{q}\right)+\frac{p_j}{q}D^2+\frac{r}{q}-2-D^2 &=0\label{eqn:first}\\
		\mbox{i.e.,}\ n_{j-1}p_{j-1}+n_{j+1}p_{j+1}+r-2q+(p_j-q)D^2 &=0.\label{eqn:second}
	\end{align}  
	
Now we have
\begin{equation}\label{eqn:negativity-of-intersection}
	\begin{split}
	0 &=n_{j-1}p_{j-1}+n_{j+1}p_{j+1}+r-2q+(p_j-q)D^2\\
	& \qquad =n_{j-1}(p_{j-1}-1)+n_{j+1}(p_{j+1}-1)+r-2\left(q-\frac{n_{j-1}+n_{j+1}}{2}\right)+((p_j-1)-(q-1))D^2\\	
	& \qquad \> n_{j-1}(p_{j-1}-1)+n_{j+1}(p_{j+1}-1)+r-2(q-1)+((p_j-1)-(q-1))D^2,
	\end{split}
\end{equation}
and thus 	
\begin{equation}\label{eqn:anti-nef}
	\begin{split}
	& n_{j-1}\left(\frac{p_{j-1}-1}{q-1}\right)+n_{j+1}\left(\frac{p_{j+1}-1}{q-1}\right)+\frac{p_j-1}{q-1}D^2+\frac{r}{q}-2-D^2\\ 											
	& \qquad \<n_{j-1}\left(\frac{p_{j-1}-1}{q-1}\right)+n_{j+1}\left(\frac{p_{j+1}-1}{q-1}\right)+\frac{p_j-1}{q-1}D^2+\frac{r}{q-1}-2-D^2\<0.
	\end{split}
\end{equation}
Observe that $B^*_{\tilde{S}}=\frac{p_{j-1}-1}{q-1} E_{j-1}+\frac{p_j-1}{q-1} D+\frac{p_{j+1}-1}{q-1} E_{j+1}+F$. Thus from \eqref{eqn:anti-nef} it follows that $(K_{\tilde{S}}+B^*_{\tilde{S}})\cdot D\<0$.	
	
\end{proof}

\begin{lemma}\cite[Proposition 3.10]{HX15}\label{lem:global-F-regularity}
	Let $k$ be a $F$-finite field of characteristic $p>5$. Let $P_1, P_2, P_3\in\mbP^1_k$ be three distinct $k$-rational points and $D_1=\frac{2}{5}P_1+\frac{2}{3}P_2+\frac{5}{6}P_3$, $D_2=\frac{1}{3}P_1+\frac{3}{4}P_2+\frac{3}{4}P_3$, $D_3=\frac{1}{2}P_1+\frac{3}{5}P_2+\frac{5}{6}P_3$ and $D_4=\sum_{i=1}^r\frac{d_i-1}{d_i}P_i$, where $r\<2$ or $r=3$ and 
	 \[
	 	(d_1, d_2, d_3)\in\{(2, 2, d), (2, 3, 3), (2, 3, 4), (2, 3, 5) \},
	 \] 
	 then $(\mbP^1_k, D_i)$ is globally $F$-regular for all $i$.\\
\end{lemma}

\begin{proof}
	For $i=1, 2$ or $3$, the global $F$-regularity of $(\mbP^1_k, D_i)$ follows by the same argument as in \cite[Proposition 3.10]{HX15}, and the global $F$-regularity of $(\mbP^1_k, D_4)$ follows from \cite[Proposition 5.3]{ST18}.
\end{proof}

\begin{lemma}\label{lem:GFR_base_change}
		Let $(X,\Delta)$ be a pair over an $F$-finite field $k$ such that $K_X+\Delta$ is $\bQ$-Cartier. Let $(X_{\ell},\Delta_\ell)$ be the base-change by some finite separable field extension $\ell/k$.  Then $(X,\Delta)$ is globally $F$ regular if and only if $(X_\ell, \Delta_\ell)$ is globally $F$-regular.
	\end{lemma}
	\begin{proof}
		Let $\msL$ be an ample line bundle on $X$, and $S=\oplus_{m\>0}H^0(X, \msL^m)$ be the section ring of $\msL$ over $k$. Then $\Spec S$ is an affine cone over $X$. Let $\Delta^S$ be the unique divisor on $\Spec S$ corresponding to $\Delta$ (see \cite[Subsection 5.2]{SS10}). Then by \cite[Proposition 5.3]{SS10}, $(X, \Delta)$ is globally $F$-regular if and only if $(\Spec S, \Delta^S)$ is strongly $F$-regular.\\
		Let $\msL_\ell$ be the pullback of $\msL$ on $X_\ell$. Then from the commutativity of the flat base-change it follows that the section ring of $\msL_\ell$ on $X_\ell$ is given by $T=\oplus_{m\>0}H^0(X, \msL^m)\otimes_k\ell$. Therefore $\Spec T=\Spec (S\otimes_k\ell)$, and hence the projection $\Spec T\to \Spec S$ is a smooth morphism of relative dimension $0$, i.e., a finite \'etale morphism, since $\ell/k$ is a finite separable extension. Let $\Delta^T_\ell$ be the unique divisor on $\Spec T$ corresponding to $\Delta_\ell$. Then by \cite[Corollary 6.31]{Schwede_Tucker}, $(\Spec T, \Delta^T_\ell)$ is strongly $F$-regular if and only if $(\Spec S, \Delta^S)$ is strongly $F$-regular.\\
	\end{proof}

\begin{remark}\label{rmk:global-F-regularity-on-P1}
	For later use, we note that if $P_1, P_2, P_3$ and $Q_1, Q_2, Q_3$ are six distinct $k$-rational points on $\mbP^1_k$ and $a_i\>0$, then $(\mbP^1_k, a_1P_1+a_2P_2+a_3P_3)$ is globally $F$-regular if and only if $(\mbP^1_k, a_1Q_1+a_2Q_2+a_3Q_3)$ is globally $F$-regular. This simply follows from the fact that these two pairs are isomorphic under a linear change of variable which takes the points $P_1, P_2, P_3$ to $Q_1, Q_2, Q_3$, respectively. 
\end{remark}

	\begin{lemma}\cite[Lemma 3.8]{HX15}
		With notations and hypothesis as in \cite[Lemma 3.8]{HX15}, the pair $(C/k, \Diff_C(B^*_{\tilde{S}}))$ is globally $F$-regular in $\chr p>5$.
	\end{lemma}

\begin{proof}
Note that being globally $F$-regular is independent of the ground field,  so replacing $k$ by a finite extension without changing $X$, we may assume that $k=H^0(C,\sO_C)$.
	We know that $(C, \Diff_C(B^c_{\tilde{S}}))$ is a klt pair which is $N$-complementary with $N\in\{1, 2, 3, 4, 6\}$ and the coefficients of $\Diff_C(B^c_{\tilde{S}})$ are in the standard set. By \autoref{lem:imperfect-curve} we know that $\deg_kK_C=-2$.  From the construction of $B^*_{\tilde{S}}$ it also follows that $\deg_k(K_C+\Diff_C(B^*_{\tilde{S}}))<\deg_k(K_C+\Diff_C(B^c_{\tilde{S}}))$; in particular, $\deg_k(K_C+\Diff_C(B^*_{\tilde{S}}))<0$, and thus $\deg_k(\Diff_C(B^*_{\tilde{S}}))<2$.
 
The possibilities for $\Diff_C(B^c_{\tilde{S}})$ are listed in the second column of \autoref{table:comp}.  We now verify that for each of these cases the corresponding $\Diff_C(B^*_{\tilde{S}})$ are globally $F$-regular.
\begin{itemize}
	\item\textbf{Case I:} If all the points in the support of $\Diff_C(B^c_{\tilde{S}})$ are $k$-rational points, then it is clear that the same proof as in \cite[Lemma 3.8]{HX15} using \autoref{lem:global-F-regularity} shows that $(C, \Diff_C(B^*_{\tilde{S}}))$ is globally $F$-regular.
	\item \textbf{Case II:} In the following we will deal with the cases where at least one point of $\Supp(\Diff_C(B^c_{\tilde{S}}))$ is not a $k$-rational point (see the second column of \autoref{table:comp}).
\end{itemize}
Now recall that $N$ is contained in the set $\{1, 2, 3, 4, 6\}$. From the second column of \autoref{table:comp} we see that for $N=6$, the support of $\Diff_C(B^c_{\tilde{S}})$ consists of $k$-rational points only. So it belongs to the Case I.

For $N\in\{1, 2, 3, 4\}$, we see from the second column of \autoref{table:comp} that $\Supp(\Diff_C(B^c_{\tilde{S}}))$ contains a point $P\in C$ such that $[k(P):k]=2, 3$ or $4$. Since $\chr p>5$, $k(P)/k$ is a separable extension. Let $\ell$ be the Galois closure of $k(P)$ over $k$. Then $C_\ell\cong \mbP^1_\ell$ and the pre-image of $P$ under $C_\ell\to C$ are $[k(P):k]$-number of $\ell$-rational points on $C_\ell$. Thus we are done either by Case I or \autoref{lem:GFR_base_change}; we explain this argument below with an explicit computation for $N=4$.  The other cases are similar.\\

If $N=4$, the possibilities are: 
 $(C, \Diff_C(B^c_{\tilde{S}}))=(\mbP^1_k, \frac{1}{2}P_1+\frac{3}{4}P_2)$ with $P_1$ a $k$-rational point and $[k(P_2):k]=2$. In this case $(C, \Diff_C(B^*_{\tilde{S}}))=(\mbP^1_k, aP_1+bP_2)$, where
 \[
 	(a, b)\in\{(1/3, 3/4); (1/2, 2/3) \}.
 \]
Since $\chr p>5$ and $[k(P_2):k]=2$, $k(P_2)/k$ is a separable extension. Let $\ell$ be the Galois closure of $k(P_2)$ over $k$. Then after a base changing to the field $\ell$ we get $(C_{\ell}, (\Diff_C(B^*_{\tilde{S}}))_{\ell})=(\mbP^1_{\ell},aP_1+b P_{2,1}+b P_{2,2})$ with \[(a,b)\in\{ (1/3 ,3/4) , (1/2,2/3)\},\]
where $P_{2, 1}$ and $P_{2, 2}$ are two $\ell$-rational points which are pre-images of $P_2\in C$ under the projection $C_\ell\to C$.
Then the global $F$-regularity of $(\mbP^1_\ell, \frac{1}{3}P_1+\frac{3}{4}P_{2, 1}+\frac{3}{4}P_{2, 2})$ follows directly from Case I, i.e., \autoref{lem:global-F-regularity}. For $(\mbP^1_\ell, \frac{1}{2}P_1+\frac{2}{3}P_{2, 1}+\frac{2}{3}P_{2 ,2})$ we compare it with $(\mbP^1_\ell, \frac{1}{2}P_1+\frac{2}{3}P_{2}+\frac{3}{4}P_{3})$. Since the latter is globally $F$-regular by \autoref{lem:global-F-regularity}, so is the former by \autoref{rmk:global-F-regularity-on-P1}, this follows from the fact that if $(X, \Delta)$ is globally $F$-regular and $0\<\Delta'\<\Delta$, then $(X, \Delta')$ is also globally $F$-regular. 
Finally we are done by \autoref{lem:GFR_base_change}.
\end{proof}

\section{Existence of flips and log minimal models}\label{sec:flips}
One of the main results of this section is the existence of pl-flips over $F$-finite fields. First we recall some definitions.

\begin{definition}\label{def:flip}
	A morphism $f:(X, \Delta)\to Z$ is called a $(K_X+\Delta)$-flipping contraction if the following conditions are satisfied: 
	\begin{enumerate}
		\item $X$ is a $\mbQ$-factorial normal $3$-fold,
		\item $f:X\to Z$ is a small projective birational contraction with $\rho(X/Z)=1$,
		\item $(X, S+B)$ has dlt singularities, and
		\item $-(K_X+\Delta)$ is ample over $Z$.
	\end{enumerate}
	
	A morphism $f^+:(X^+, \Delta^+)\to Z$ is called a flip (if it exists) of the flipping contraction $f:(X, \Delta)\to Z$ if the following conditions are satisfied:
	\begin{enumerate}
		\item $X^+$ is a $\mbQ$-factorial normal $3$-fold,
		\item the induced birational map $\phi:X\bir X^+$ is an isomorphism in codimension $1$,
		\item $f^+:X^+\to Z$ is a small projective birational contraction with $\rho(X^+/Z)=1$,
		\item $(X^+, \Delta^+)$ has dlt singularities, where $\Delta^+=\phi_*\Delta$, and
		\item $K_X+\Delta$ is ample over $Z$.
	\end{enumerate} 
\end{definition}

\begin{definition}\label{def:pl-flip}
	A morphism $f:(X, S+B)\to Z$ is called a \emph{pl-flipping} contraction if the following conditions are satisfied:
	\begin{enumerate}
		\item $f:X\to Z$ is a small projective birational contraction with $\rho(X/Z)=1$,
		\item $(X, S+B)$ has plt singularities with $\lrd S+B\rrd=S$ irreducible, and
		\item $-S$ and $-(K_X+S+B)$ are ample over $Z$.
	\end{enumerate} 
	A flip (if it exists) of a pl-flipping contraction is called a \emph{pl-flip}.
\end{definition}

\begin{lemma}
	Let $(X, S+B\>0)$ be a plt $3$-fold pair over an $F$-finite field of characteristic $p>5$. Assume that the coefficients of $B$ are in the standard set $\mcS=\{1-\frac{1}{n}:n\in\mbN\}\cup\{1\}$ and $\lrd S+B\rrd=S$ is irreducible. Let $S^n\to S$ be the normalization, and $K_{S^n}+B_{S^n}=(K_X+S+B)|_{S^n}$ is by adjunction.  Then $S$ is normal.
	\end{lemma}
\begin{proof}
	By adjunction $(S^n,\Delta_{S^n})$ has klt singularities.  Now by \cite[Theorem 1.2]{ST18}, $(S^n, B_{S^n})$ is strongly $F$-regular. Then by the same proof as in \cite[Pro. 4.1]{HX15} or \cite[Cor. 5.4]{Das15} it follows that $S$ is normal.
	\end{proof}

The main result of this section is the following theorem.
\begin{theorem}\label{thm:pl-flip}\cite[Theorem 4.12]{HX15}
	Fix an $F$-finite field $k$ of characteristic $p>5$. Let $f:(X, S+B)\to Z$ be a pl-flipping contraction of projective $3$-folds defined over $k$ and the coefficients of $B$ belong to the standard set $\mcS=\{1-\frac{1}{n}:n\in\mbN\}\cup\{1\}$. Then the flip of $f$ exists. 
\end{theorem}

\begin{proof}
The arguments of \cite[Thm 4.12]{HX15} go with minor changes over $F$-finite (infinite) fields $k$ of char $p>5$ by using \autoref{thm:globally-f-regular} .  The changes necessary are as follows:

\begin{enumerate}
	\item Whenever using Bertini's theorem we must assume that the base field $k$ is infinite in order for general $k$-hyperplanes to exist.  This is harmless because the result is already known over finite fields by \cite{GNT15}.
	\item The proof \cite[Lemma 4.2]{HX15} uses the inequality between the usual Seshadri constant and the $F$-Seshadri constant, which was proved over algebraically closed field in \cite[Proposition 2.12]{mustata_frobenius_2012}.  As usual with $F$-singularity arguments, these arguments go through for an $F$-finite base field.  However, we could not find suitable references for standard properties of the usual Seshadri constant over arbitrary fields, so we include proofs of these in \autoref{sec:seshadri}.
	\item The proofs of Lemma 4.4, 4.5, 4.6, 4.7, 4.8, 4.9, 4.10 and Corollary 4.11 of \cite{HX15} work over $F$-finite fields without any change.
\end{enumerate}

\end{proof}

\subsection{Special termination}
Special termination holds for dlt pairs as in the characteristic zero case. We need the following definition.
\begin{definition}\label{def:flipping-flipped-locus}
	Let $f:X\to V$ be a flipping contraction and $f^+:X^+\to V$ be its flip (or generalized flip). Then the exceptional locus $\Ex(f)$ is called the \emph{flipping locus} and $\Ex(f^+)$ the \emph{flipped locus}. 
\end{definition}
\begin{theorem}\cite[Pro. 5.5]{Bir16}\label{thm:special-termination}
	Let $(X,B)$ be a projective $\bQ$-factorial dlt pair of dimension $3$ over an arbitrary field $k$ of characteristic $p>5$. Consider a sequence of log flips  starting from $(X, B)=(X_0, B_0)$:
	\[
		(X_0, B_0)\bir (X_1, B_1)\bir (X_2, B_2)\bir\cdots (X_i, B_i)\bir\cdots,
	\]
where $\phi_i:X_i\to Z_i$ is a flipping contraction and $\phi_i^+:X_i^+=X_{i+1}\to Z_i$ is the flip. Then there exists a positive integer $i_0>0$ such that the flipping locus (and thus the flipped locus) is disjoint from $\lrd B_i\rrd$ for all $i\>i_0$.	
	
\end{theorem}

\begin{proof}
	The same proof as in \cite[Proposition 5.5]{Bir16} works here without any change. 
\end{proof}

\subsection{Dlt flips with standard coefficients over $F$-finite fields}
\begin{theorem}\label{thm:dlt-flips-F-finite-S}
	Let $f:(X, B)\to Z$ be a flipping contraction, where $(X, B)$ is a $\mbQ$-factorial dlt pair, and $X$ is a projective 3-fold defined over a $F$-finite field of char $p>5$ and $Z$ is a quasi-projective variety. Furthermore, assume that the coefficients of $B$ belong to the set $\{1-\frac{1}{n}: n\in\mbN\}\cup\{1\}$. Then the flip of $f$ exists.
\end{theorem}

\begin{proof}
	The same proof as \cite[Theorem 1.1]{HX15} holds here, using \autoref{thm:cone_theorem} and \autoref{thm:pl-flip}.
\end{proof}

\begin{definition}\label{def:generalized-flip}
	Let $(X, \Delta)$ be a projective $\mbQ$-factorial dlt $3$-fold pair. Let $f:X\to V$ be a proper small contraction, i.e., $f_*\mcO_X=\mcO_V$, of a $(K_X+\Delta)$-negative extremal ray $R$ to an algebraic space $V$. This is called a \emph{generalized flipping contraction}. The \emph{generalized flip} of $f$ is a proper small contraction $f^+:X^+:\to V$  from a normal projective $\mbQ$-factorial $3$-fold $X^+$ such that the induced birational map $\phi:X\bir X^+$ is isomorphism in codimension $1$ and $K_{X^+}+\Delta^+$ is $f^+$-nef, where $\Delta^+:=\phi_*\Delta$. 
\end{definition}

\begin{lemma}\label{lem:algebraic-space-to-varieties}
	Let $(X, \Delta)$ be a projective $\mbQ$-factorial dlt $3$-fold pair over an $F$-finite field of char $p>5$. Let $f:X\to V$ be a birational contraction of a $(K_X+\Delta)$-negative extremal ray $R$ to an algebraic space $V$. Further assume that $S$ is a component of $\lrd\Delta\rrd$ such that $S$ is normal and $S\cdot R<0$. Then $V$ is a projective variety. Moreover, if $L'$ is a $\mbQ$-Cartier divisor on $X$ such that $L'\cdot R=0$, then $L'\sim_\mbQ f^*D'$ for some $\mbQ$-Cartier divisor $D'$ on $V$.
\end{lemma}

\begin{proof}
 Let $H$ be an ample $\mbQ$-divisor on $X$ such that $L=K_X+\Delta+H$ is nef and big and $L^\bot=R$; in particular, $L\num_V 0$. Note that in order prove that $V$ is an algebraic variety it is enough to show that $L$ semi-ample. To that end by adjunction we have $(S, \Delta_S)$ is dlt and $L|_S=K_S+\Delta_S+H|_S$ is nef on $S$, where $K_S+\Delta_S=(K_X+\Delta)|_S$. Since $(S, \Delta_S)$ is dlt and $H|_S$ is ample, by passing to a log resolution of $(S, \Delta_S)$ and using Bertini theorem we can find an effective $\mbQ$-divisor $A\sim_\mbQ H|_S$ such that $(S, \Delta_S+A)$ is dlt. Then by \cite[Theorem 1.1]{Tan15} $K_S+\Delta_S+A$ is semi-ample. In particular, $L|_S\sim_\mbQ K_S+\Delta_S+A$ is semi-ample.\\
 Now since $S\cdot R<0$, $\mbE(L)$ is contained in $S$ and $\mbE(L)=\mbE(L|_S)$. Thus $L|_{\mbE(L)}=(L|_S)|_{\mbE(L|_S)}$ is semi-ample, since $L|_S$ is semi-ample. Then by \cite[Theorem 0.2]{Kee99} $L$ is semi-ample on $X$.
 
 Let $m$ be sufficiently large that $K_X+\Delta+mL$ is big.
 For the second part, by the cone theorem for effective pairs \autoref{thm:cone_theorem}, let $W_R=\NE(X)_{K_X+\Delta+mL\>0}+\sum_{R_i\neq R} R_i$. Then $W_R$ is a closed cone in $\mbox{N}_1(X)$ and $L$ is positive on $W_R\setminus\{0\}$. Fix a norm $||\cdot||$ on $\mbox{N}_1(X)$ and let $S$ be the unit sphere in $\mbox{N}_1(X)$ centered at the origin, i.e., $S=\{\gamma\in\mbox{N}_1(X)\;:\; ||\gamma||=1\}$. Then $S\cap W_R$ is a compact set and thus $L'$ takes a minimum value on $S\cap W_R$. In particular, $L'+mL$ is positive on $W_R$ for all $m\gg 0$. Consequently, $L'+mL$ is nef and big, and $(L'+mL)^\bot=R$ for $m\gg 0$. Then by a similar proof as in the previous case it follows that $L'+mL$ is semi-ample for $m\gg 0$. Since $(L'+mL)^\bot=L^\bot=R$,  by the rigidity lemma (see \cite[Proposition 1.14]{Deb01}) they induces the same contraction $f$. In particular, there is an ample divisor $D''$ on $V$ such that $L'+mL\sim_\mbQ f^*D''$, i.e., $L'\sim_\mbQ f^*(D''-M)$, where $mL\sim_\mbQ f^*M$.
 
\end{proof}

\begin{corollary}\label{cor:generalized-to-usual-flips}
Let $(X, \Delta)$ be a projective $\mbQ$-factorial dlt $3$-fold pair over an $F$-finite field of char $p>5$. Let $f:X\to V$ be a small contraction of a $(K_X+\Delta)$-negative extremal ray $R$ to an algebraic space $V$. Further assume that $S$ is a component of $\lrd\Delta\rrd$ such that $S$ is normal and $S\cdot R<0$. Then the generalized flip (if it exists) of $f$ is the same as the usual flip of $f$.
\end{corollary}

\begin{proof}
First note that $V$ is a projective variety by \autoref{lem:algebraic-space-to-varieties}. Let $f^+:(X^+, \Delta^+)\to V$ be the generalized flip of $f$. Then by a similar proof as in \cite[Proposition 3.37]{KM98} using \autoref{lem:algebraic-space-to-varieties} we see that $X^+$ is $\mbQ$-factorial, $(X^+, \Delta^+)$ is dlt and the relative Picard number $\rho(X^+/V)=\rho(X/V)=1$. One thing that is not obvious is whether $K_{X^+}+\Delta^+$ is $f^+$-ample or not. We will give a short proof of this fact here. To the contrary assume that $K_{X^+}+\Delta^+\num_V 0$. Let $H^+$ be an ample divisor on $X^+$ and $H$ be its strict transform on $X$. We claim that $H\cdot R\neq 0$, indeed if $H\cdot R=0$, then by \autoref{lem:algebraic-space-to-varieties} $H\sim_\mbQ f^*D$ for some $\mbQ$-Cartier divisor $D$ on $V$. Then by pushing forward to $X^+$ we get $H^+\sim_\mbQ {f^+}^*D$, a contradiction to the fact that $H^+$ is ample. Therefore $H\cdot R\neq 0$; in particular, $K_X+\Delta\num_V aH$ for some $a\neq 0$. Then by \autoref{lem:algebraic-space-to-varieties} pushing forward to $X^+$ gives $K_{X^+}+\Delta^+\num_V aH^+$, but this is a contradiction, since $K_{X^+}+\Delta^+\num_V 0, a\neq 0$ and $H^+$ is ample. 
\end{proof}

\subsection{Existence of generalized flips}
\begin{theorem}\label{thm:HX-generalized-flips}
	Let $k$ be a $F$-finite field of characteristic $p>5$ and $(X, \Delta)$ a projective $\mbQ$-factorial threefold klt pair over $k$ such that $K_X+\Delta$ is pseudo-effective and all coefficients of $\Delta$ are in the standard set $\{1-\frac{1}{n}:\,n\in\mbN\}$. Let $R$ be a $(K_X+\Delta)$-negative extremal ray and $f:X\to Z$ the corresponding proper birational contraction to a proper algebraic space (given by \autoref{thm:contraction_theorem}) such that a curve $C$ is contracted if and only if $[C]\in R$. Then
	\begin{enumerate}
		\item the generalized flip (see \autoref{def:generalized-flip}) of $f$ exists, and
		\item if $f$ is a divisorial contraction, i.e., $\codim_X\Ex(f)=1$, then $X^+=Z$ and in particular $Z$ is projective.
	\end{enumerate}
\end{theorem}

\begin{proof}
	The same proof as in \cite[Theorem 5.6]{HX15} holds using \autoref{thm:pl-flip} and \ref{thm:special-termination}. 
\end{proof}

\subsection{Generalized flips with arbitrary coefficients over $F$-finite fields}

\begin{theorem}\label{thm:dlt-flips-F-finite}
	Let $f:(X, B)\to Z$ be a flipping contraction, where $(X, B)$ is a $\mbQ$-factorial dlt pair, and $X$ is a $3$-fold defined over a $F$-finite field of char $p>5$ and $Z$ is an algebraic space. Then the generalized flip of $f$ exists.
\end{theorem}

\begin{proof}
	The same proof as in \cite[Theorem 6.3]{Bir16} holds using \autoref{thm:dlt-flips-F-finite-S} and \autoref{thm:HX-generalized-flips}.
\end{proof}

\begin{theorem}\label{thm:dlt-flips-for-arbitrary-coefficients}
	Let $(X, B)$ be a $\mbQ$-factorial dlt pair of dimension $3$ over an $F$-finite field $k$ of char $p>5$. Let $X\to Z$ be a $(K_X+B)$-negative extremal flipping projective contraction. Then the flip of $f$ exists.
\end{theorem}
\begin{proof}
	The same proof as in \cite[Theorem 1.1]{Bir16} holds using \autoref{thm:dlt-flips-F-finite} and \autoref{pro:log_minimal_models_f_finite}.
\end{proof}

\subsection{Weak Zariski decomposition and LMMP} 
Let $D$ be an $\mbR$-Cartier divisor on a normal variety $X$ defined over an $F$-finite field of char $p>5$. Let $X\to Z$ be projective contraction over $k$. A \emph{weak Zariski decomposition}$/Z$ for $D$ consists of a projective birational morphism $f:W\to X$ from a normal variety, and a numerical equivalence $f^*D\num_Z P+M$ such that 
\begin{enumerate}
	\item $P$ and $M$ are $\mbR$-Cartier divisors,
	\item $P$ is nef$/Z$ and $M\>0$.
\end{enumerate}

\begin{proposition}\label{pro:weak-zariski-to-mmp}
	Let $(X, B)$ be a projective log canonical pair of dimension $3$ defined over an $F$-finite field $k$, and $X\to Z$ a projective contraction. Assume that $K_X+B$ has a weak Zariski decomposition$/Z$. Then $(X, B)$ has a log minimal model over $Z$.
\end{proposition}
\begin{proof}
	The same proof as in \cite[Proposition 8.3]{Bir16} works here.
\end{proof}

\begin{proposition}\label{pro:log_minimal_models_f_finite}
	Let $(X,B)$ be a quasi-projective klt pair of dimension $3$ over an $F$-finite field $k$ of char $p>5$. Let $f:X\to Z$ be a projective contraction. If  $K_X+B$ is pseudo-effective over $Z$, then $(X,B)$ has a log minimal model over $Z$.
\end{proposition}
\begin{proof}
	Same proof as in \cite[Theorem 1.2]{Bir16} works.
\end{proof}

\subsection{Arbitrary fields}

We may now reduce the case of an arbitrary field to the $F$-finite case in some of our results via the following lemma.

\begin{lemma}\label{lem:model}
	Let $X$ and $Y$ be quasi-projective varieties over an arbitrary field $k$, and $f:X\to Y$ a morphism and $D$ a Cartier divisor on $X$.  Then there exists an $F$-finite subfield $\ell\subset k$ and varieties $X_{\ell}$, $Y_{\ell}$, a morphism $f_{\ell}$, and a Cartier divisor $D_{\ell}$ on $X$ such that $X\cong X_{\ell}\otimes_\ell k $, $Y\cong Y_{\ell}\otimes_\ell k$, $f=f_{\ell}\otimes_\ell k$ and $D=D_\ell\otimes_\ell k$. 
	\end{lemma}
\begin{proof}
	Fix an embedding $X\hookrightarrow \mathbb{P}_k^n$, and take $\ell$ to be the subfield of $k$ generated over $\mathbb{F}_p$ by all coefficients of the equations of $X$ and $\mathbb{P}_k^n\setminus X$ in $\mathbb{P}^n_k$.  Then $X_{\ell}$ is defined by the same equations in $\mathbb{P}_{\ell}^n$.  We construct $Y_{\ell}$ and $f_{\ell}$ and $D_{\ell}$ in a similar way, possibly enlarging the field $\ell\subset k$ if necessary.  Note that $\ell $ is $F$-finite because it is finitely generated over a perfect field.
	\end{proof}

\begin{lemma}\label{lem:klt-descends}
	Let $(X,B\>0)$ be a pair over an arbitrary field $k$. Let $\ell$ be an $F$-finite subfield of $k$ and  $(X_{\ell},B_{\ell})$ a pair such that $(X,B)=(X_{\ell},B_{\ell})\otimes_{\ell} k$.  If $(X,B)$ is klt, then $(X_{\ell},B_{\ell})$ is also klt.
	\end{lemma}
\begin{proof}
	Since $X$ is normal and $X\to X_\ell$ is a faithfully flat morphism, $X_{\ell}$ is also normal by \cite[Lemma 033G]{STP}.  If $\pi:X\to X_{\ell}$ is the induced morphism, then $K_X+B=\pi^*(K_{X_{\ell}}+B_{\ell})$ is $\bR$-Cartier. Let $f_{\ell}:Y_{\ell}\to X_{\ell}$ be a proper birational morphism from a normal variety $Y_\ell$ and 
	 \[
	 	K_{Y_{\ell}}+B_{Y_{\ell}}=f_{\ell}^*(K_{X_{\ell}}+B_{\ell}).
	 \]

We need to show that the coefficients of $B_{Y_{\ell}}$ are strictly less than $1$. Let $f:Y\to X$ be the morphism $f\otimes_{\ell} k$, where $Y=Y_\ell\otimes_\ell k$ and $X=X_\ell\otimes_\ell k$.  Note that $f:Y\to X$ is birational. We need to show that $Y$ is an integral scheme. First we claim that $Y$ is irreducible. Indeed, let $U_\ell\subset X_\ell$ and $V_\ell\subset Y_\ell$ be isomorphic open sets via $f_\ell$. Then $V=V_\ell\otimes_\ell k\subset Y$ is isomorphic to $U=U_\ell\otimes_\ell k\subset X$. Since $U$ is irreducible, $V$ is also irreducible, and hence $Y$ is irreducible. Now we will show that $Y$ is reduced. Indeed, since $Y_\ell$ is $S_1$ and being $S_1$ is stable under base change, $Y$ is also $S_1$. Moreover, since $Y$ is irreducible and an open subset of $Y$ is isomorphic to an open subset of $X$ and $X$ is integral, it follows that $Y$ satisfies $R_0$.\\

Now back to the original proof, if $Y$ is not normal, then replacing $Y$ by its normalization, we get the following log equation 
\begin{equation}\label{eqn:base-change-log-equation}
	K_Y+B_Y=f^*(K_X+B)=f^*\pi^*(K_{X_{\ell}}+B_{\ell})=\pi_Y^*(K_{Y_\ell}+B_{Y_{\ell}}),
\end{equation}
where $\pi_Y:Y\to Y_\ell$ is the induced morphism.\\
Since $(X, B)$ is klt, the coefficients of $B_Y$ are strictly less than $1$. Moreover, since the coefficients of $B_{Y_\ell}$ can only increase via pullback by $\pi_Y$, it follows from \eqref{eqn:base-change-log-equation} that the coefficients of $B_{Y_\ell}$ are strictly less than $1$.
	\end{proof}

\begin{lemma}\label{lem:Q-factorial-descends}
	Let $Y$ be a normal variety defined over a field $K$. Let $k\subset K$ be an infinite subfield of $K$ and $X$ is a variety defined over $k$ such that $Y=X\otimes_k K$. If
	$Y$ is $\mbQ$-factorial, then $X$ is also $\mbQ$-factorial.
	
\end{lemma}

\begin{proof}
	Let $D$ be a prime Weil divisor on $X$ and $D_Y$ its flat pullback on $Y$. Since $Y$ is $\mbQ$-factorial, there is a positive integer $m>0$ such that $mD_Y$ is Cartier. Now fix a closed point $x\in X$ and let $U=\Spec A$ be an affine open subset of $X$ containing $x$. We will show that there is an open subset $x\in V\subset U$ such that $mD|_V$ is a principal divisor. To that end, replacing $X$ by $U$ and $Y$ by $f^{-1}U=\Spec (A\otimes_k K)$ we may assume that $X$ and $Y$ are both affine. Note that even though $mD_Y$ is Cartier on $Y=\Spec (A\otimes_k K)$, it doesn't imply that $mD_Y$ is a principal divisor. However, $Y$ can further be covered by finitely many open sets $\{V_i\}$ such that $mD_Y|_{V_i}$ is  principal. Let $L$ be the field generated over $k$ by the coefficients of the local equations of $mD_Y$ on $V_i$. Then $Y$ and $mD_Y$ are both defined over the field $L$.  Therefore if we let $X_{L}:=X\otimes_k L$ we may assume that $mD_{X_{L}}$ is Cartier. Let $t_1, t_2,\ldots, t_m\in L$ be a transcendence basis of $L$ over $k$. Then $L$ is a finite extension of $\ell=k(t_1, t_2,\ldots, t_m)$. Thus $X_L\to X_\ell=\Spec (A\otimes\otimes_k\ell)$ is a finite morphism. Then by the argument of \cite[Proposition 3.3]{Wal15} it follows that some multiple of $m D_{X_{\ell}}$ is Cartier.  Note that the argument of \cite[Proposition 3.3]{Wal15} applies to one divisor at a time, so we do not need $X_{L}$ to be $\mathbb{Q}$-Cartier to apply it.  
	So replacing $Y$ by $X_\ell$ and $K$ by $\ell$ we may assume that $K=k(t_1, t_2,\ldots, t_m)$.\\
	Now let $\pi:Y\to X$ be the induced morphism. Then the fiber $\pi^{-1}(x)$ is given by $\Spec (k(x)\otimes_k K)$. But since $x\in X$ is a closed point, $k(x)$ is a finite extension of $k$. Therefore by \cite{Sha77}, the Krull dimension of the ring $k(x)\otimes_k K$ is given by 
	\[\dim (k(x)\otimes_k K)=\min\{\mbox{tr. deg.}(k(x)/k), \mbox{tr. deg.}(K/k)\}=0.\]
	 In particular, the fiber $\pi^{-1}(x)$ is a $0$-dimensional closed subscheme of (a Noetherian scheme) $Y$, and hence $\pi^{-1}(x)$ is a finite set.   Furthermore, it is known that for any field extension $F/k$, $F\otimes_kk(t_1,\ldots,t_m)$ is the subring of $F(t_1,\ldots, t_m)$ consisting of all elements of the form $\{\frac{f}{g}:f\in F[t_1,\ldots,t_m],g\in k[t_1,\ldots,t_m]\}$, by the universal property of $\otimes$.  Hence it is an integral domain, and in our case $k(x)\otimes_k K$ is a field, since it is $0$-dimensional. In particular, $\pi^{-1}(x)$ consists of a single point.

	Suppose that $B:=A\otimes_k K$ and let $y=\pi^{-1}(x)$. Since $mD_Y$ is a Cartier divisor on $Y$, there exists a basic open neighborhood $\Spec B_g\subset Y$ of $y$ for some $g\in B$ such that $mD_Y|_{\Spec B_g}$ is principal. Let $I(mD_Y)$ denote the ideal of the closed subscheme $mD_Y\subset Y$. Then $I(mD_Y)_g=(\vphi)$ for some $\vphi\in B_g$. Note that $g(y)\ne 0$. 
	We claim that there is a $h\in A$ such that $h(x)\ne 0$ and $mD|_{\Spec A_h}$ is a principal. Assume that the ideal of $mD\subset X$ is generated by $a_1, a_2,\ldots, a_n\in A$, i.e. $I(mD)=(a_1, a_2,\ldots, a_n)$. Then the extension of $I(mD)$ in $B$ is generated by $a_1\otimes 1,\ldots, a_n\otimes 1$. Note that this extension is precisely $I(mD_Y)\subset B$, i.e. $I(mD_Y)=(a_1\otimes 1,\ldots, a_n\otimes 1)$. Therefore $(a_1\otimes 1,\ldots, a_n\otimes 1)_g=(\vphi)$ in $B_g$. Now let $\vphi=\frac{f}{g^{n_0}}$ for some $n_0\>0$, where $f\in B$. Then $g^{n_0+n_i}(a_i\otimes 1)=fh_i$ for some $n_i\>0$ and $h_i\in B$, $i=1,2,\ldots, n$. Let $f=\sum_jb_j\otimes \frac{r_j}{s_j}$, $g=\sum_j c_j\otimes\frac{u_j}{v_j}$ and $h_i=\sum_{k} d_{ik}\otimes \frac{p_{ik}}{q_{ik}}$, where $b_j, c_j, d_{ik}\in A$ and $r_j,s_j, u_j, v_j, p_{ik}, q_{ik}\in k[t_1,\ldots,t_m]$. Since $k$ is infinite, there is a point $P=(\lambda_1,\ldots,\lambda_m)\in\mbA^m_k$ such that $s_j, v_j$ and $q_{ik}$ all take non-zero values at $P$.  
	Then after clearing the denominators and replacing $p_{ik}, r_j$ and $u_j$ by $p'_{ik}, r'_j$ and $u'_j$ respectively, 	
	 we have:
	\[ \Big(\prod_{j}v_j^{n_0+n_i}\Big)\Big(\sum_{k} d_{ik}\otimes p'_{ik}\Big)\Big(\sum_j(b_j\otimes r_j')\Big)=(a_i\otimes 1)\Big(\sum_j c_j\otimes u'_j\Big)^{n_0+n_i} \prod_js_j  \prod_k q_{ik}\]
	Now evaluating the polynomials in the variables $t_1, t_2,\ldots, t_m$ in the expression above at the point $P=(\lambda_1,\ldots,\lambda_m)$ shows that $\left(\sum_ju'_j(P)c_j\right)a_i$ is contained in the principal ideal generated by $\sum_j r_j'(P)b_j$ for all $i=1,2,\ldots, n$. Set $h:=\sum_ju'_j(P)c_j\in A$; then $h(x)\neq 0$, since $g=\sum_j c_j\otimes \frac{u_j}{v_j}$ and $g(y)\ne 0$.
	
	Then clearly $a_i \in\left(\sum_j r_j'(P)b_j\right)$ in the ring $A_h$ for all $i=1,2\ldots, n$, i.e. $I(mD)_h=\left(\sum_j r_j'(P)b_j\right)$ in $A_h$. Hence $mD|_{\Spec A_h}$ is a principal.

\end{proof}

\begin{proof}[Proof of \autoref{thm:flip}]
	We will show that $\oplus_{m\>0}f_*\mcO_X(\lrd m(K_X+B)\rrd)$ is a finitely generated $\mcO_Z$-algebra. To that end first we assume that $Z=\Spec A$ is affine. Moreover, since $X$ is $\mbQ$-factorial, by perturbing the coefficients of $B$ we may assume that $B$ is a $\mbQ$-divisor, $(X, B)$ is dlt and $-(K_X+B)$ is stil $f$-ample. Let $d$ be the Cartier index of $K_X+B$. Then it is enough to prove that $\oplus_{m\>0}H^0(X, \mcO_X(md(K_X+B)))$ is a finitely generated $A$-algebra.\\ 
Let $\ell$ be a finitely generated field over $\mbF_p$ containing the coefficients of a set of equations defining $X, B, Z$, the morphism $f$, and the flipping curves of $f$. Let $X_\ell, B_\ell, A_\ell$ and $f_\ell$ are the respective $\ell$-models such that $X=X_\ell\otimes_\ell k, B=B_\ell\otimes_\ell k, A=A_\ell\otimes_\ell k$, and $f=f_\ell\otimes_\ell k$. Then from Lemma \ref{lem:model}, \ref{lem:klt-descends}, and \ref{lem:Q-factorial-descends} it follows that $f_\ell:X_\ell\to Z_\ell$ is a $(K_{X_\ell}+B_{\ell})$-flipping contraction. Since $\ell$ is $F$-finite, by \autoref{thm:dlt-flips-for-arbitrary-coefficients} the flip of $f_\ell$ exists, i.e., $\oplus_{m\>0}H^0(X_\ell, \mcO_{X_\ell}(md(K_{X_\ell}+B_\ell)))$ is a finitely generated $A_\ell$-algebra.\\
Let $\pi:X\to X_\ell$ be the projection morphism. Then we have $K_X+B=\pi^*(K_{X_\ell}+B_\ell)$. Also notice that $\Spec A\to \Spec A_\ell$ is a flat morphism, since $A=A_\ell\otimes_\ell k$. Thus by flat base change and the fact that $X=X_\ell\otimes_{A_\ell}A$ we have 
\[
	H^0(X, \mcO_X(md(K_X+B)))=H^0(X_\ell, \mcO_{X_\ell}(md(K_{X_\ell}+B_\ell)))\otimes_{A_\ell} A.
\]
In particular, $\oplus_{m\>0}H^0(X, \mcO_X(md(K_X+B)))$ is a finitely generated $A$-algebra.
	\end{proof}

\begin{proof}[Proof of \autoref{thm:bpf}]
	First assume that $k$ is $F$-finite.  Then the same proof as in \cite[Theorem 1.4]{Bir16} works here.  Now for a general field $k$, we can use \autoref{lem:model} to take a finitely generated but infinite extension $\ell$ of $\mathbb{F}_p$ over which $X$,$B$, $D$ and $Z$ and $f$ all have $\ell$-models.  Then $D_{\ell}$ is semi-ample over $Z_{\ell}$ by the $F$-finite case, from which it follows that the pullback $D$ is semi-ample over $Z$.
\end{proof}

\begin{proof}[Proof of \autoref{thm:divisorial-contraction}]
	This follow from the same proof as \cite[Theorem 1.5]{Bir16}.
\end{proof}

\begin{proof}[Proof of \autoref{thm:log-minimal-model}]
	This is \autoref{pro:log_minimal_models_f_finite}.
\end{proof}

\begin{remark}
	The proofs referenced in \autoref{thm:divisorial-contraction} and \autoref{thm:log-minimal-model} above made use of a projective resolution of singularities in order to apply arguments similar to Shokurov's pl-flip reduction, see Remark \ref{rmk:on-projective-resolution}.  If we had such a resolution over an arbitrary field then all of the above arguments would go through in that situation once we have reduced the general case of \autoref{thm:flip} to the $F$-finite case as above.  {The restriction that the ground field is $F$ finite has been removed in \autoref{thm:divisorial-contraction} in \cite[Corollary 9.16]{seven_authors} via different arguments}.
	\end{remark}

\section{Cone theorem II}\label{sec:cone2}

\begin{lemma}
	Let $(X,B)$ be a $\bQ$-factorial projective dlt $3$-fold pair over an $F$-finite field $k$ such that $B$ is a $\bQ$-divisor and $K_X+B$ is not nef.  Then there is a natural number $n$ depending only on $(X,B)$ such that if $H$ is an ample Cartier divisor and \[\lambda=\min\{t\>0\;|\; K_X+B+tH \mathrm{\ is\ nef}\},\] then $\lambda=\frac{n}{m}$ for some natural number $m$. Moreover, there is a curve $C$ and a positive integer $d_{C}$ depending only on $X, C$ and the ground field $k$  such that: 
	\begin{enumerate}
	\item\label{item:extremal-bound} $0<-(K_X+B)\cdot_k C\leq  6d_{C}$,
	\item  If $L$ is any Cartier divisor on $X$, then $L\cdot_k C$ is divisible by $d_{C}$, and
	\item $(K_X+B+\lambda H)\cdot C=0$.
	\end{enumerate}
	\end{lemma}
\begin{proof}
	We follow the proof of \cite[Lemma 3.2 and 3.3]{BW17}. It is enough to find the curve $C$ satisfying the inequality \autoref{item:extremal-bound}; the rest follows from replacing $C$ by $\frac{C}{d_C}$ in the usual argument as in the proof \cite[Lemma 3.2 and 3.3]{BW17}.
	
	\emph{Case 1:} $L=K_X+B+\lambda H$ big (see \cite[3.2]{BW17}).  As in \cite[3.2]{BW17} we may assume that we have an extremal ray $R$ and $L\cdot R=0$ and big and nef divisor $N$ with $N^{\perp}=R$.  We have shown that such rays have projective contractions.  If $R$ gives a divisorial contraction, the argument is the same as \cite[3.2]{BW17}, with only the minor change that the family of curves obtained need not be $\mathbb{P}^1_k$, and in all intersection statements we replace a curve $\Gamma$ with $\frac{\Gamma}{d_{\Gamma}}$.
	
	The flipping case is a slightly more involved argument, but still goes through more or less unchanged after replacing $\Gamma$ with $\frac{\Gamma}{d_{\Gamma}}$.  The following facts are used in the argument:
	\begin{enumerate}
		\item If $P^+$ is an ample divisor on $X^+$ and $\psi:W\to X^+$ is a projective morphism, then given any integer $N$, we can replace  $P^+$ with a sufficiently large multiple that for any curve $\Gamma$ on $W$, $\psi^*P^+\cdot\Gamma$ is an integer divisible by $Nd_{\Gamma}$.
		\item If a morphism $f:Y\to X$ is an isomorphism near the generic point of a curve $\Gamma_Y$ on $Y$, then  $\Gamma_Y$ and $f(\Gamma_Y)$ have the same associated constant $d$.
	\end{enumerate}
	
	\emph{Case 2:} $L=K_X+B+\lambda H$ not big.
	
	Let $K$ be an uncountable algebraically closed extension of $k$ (which is therefore an uncountable algebraically closed extension of $\overline{k}$).  Then let $\pi:Y\to X$ be the normalization of the maximal reduced subscheme of $X\otimes_k K$, and let $H_Y=\pi^*H$ and $L_Y=\pi^*L$.  As $Y$ is a variety over an uncountable algebraically closed field, there is a nef reduction map $f:Y\dashrightarrow Z$ of $L$ (see \cite[Theorem 2.9]{CTX15}).  Also, by \cite{Tan15f} there is some effective divisor $\Delta$ on $Y$ such that $\pi^*(K_X+B)=K_Y+\Delta$.  
	
	First assume $Z$ has positive dimension.  
	Let $\phi:W\to X$ be the induced morphism obtained by resolving the singularities of the graph of $Y\bir Z$ so that $f:W\to Z$ is a morphism.  Continue to run the argument of \cite[Proof of Lemma 3.3, Paragraph 3]{BW17}, using $Y$ in place of $X$.  Note that in that argument it does not matter that $\Delta$ is not a boundary, for $\Theta$ is already not necessarily a boundary.  We can simply choose $P$ sufficiently generally that $\Delta$ does not contain some component of $G$ in its support.  We obtain a family of curves $\{C_\alpha\}_{\alpha\in\Lambda}$ such that 
	\[-3\leq -(K_Y+\Delta)\cdot C_\alpha<0.\]
	
	The required inequality then follows from \autoref{lem:inequality}.

	Now assume that $Z$ has dimension zero.  This implies that we already had $L\num 0$, that is $-(K_X+B)\num \lambda H$ is ample.

	Choose a general smooth projective curve $\Gamma$ in $Y$ (by cutting with general hyperplanes), which is not in the support of $\Delta$.  Note that this forces $\Gamma\cdot K_Y<0$.  Then choose a point $c\in \Gamma$, also away from the support of $\Delta$.  By \cite[Theorem II.5.8]{kollar_rational} there is a curve $C_Y$ which passes through $c$ and satisfies
	\[ -(K_Y+\Delta)\cdot C_Y\leq 6 \frac{-(K_Y+\Delta)\cdot C_Y}{-K_Y\cdot C_Y}\leq 6\frac{-(K_Y+\Delta)\cdot C_Y}{-(K_Y+\Delta)\cdot C_Y}=6.\]
	
Again, the required inequality follows from \autoref{lem:inequality}.
	
	\end{proof}

\begin{proof}[Proof of \autoref{thm:cone}]
	This is now as in \cite[1.1]{BW17} using \cite[Theorem 3.15]{KM98}, with the appropriate modifications for the additional constant $d_C$, which works exactly as in \autoref{sec:keel}.
	\end{proof}

\section{Corollaries}\label{sec:corollary}
\begin{proof}[Proof of \autoref{cor:glmm}]
	Let $(X,B)$ be the generic fiber of $f:\sX\to Z$.  We know that $(X,B)$ has a log minimal model $(Y,B_Y)$.  Furthermore, let $\phi_W:W\to X$ be a log resolution of $(X,B)$ which has a projective morphism to $Y$.  Now let $\sY\to Z$ and $\sW\to Z$ be projective morphisms whose generic fibers are $Y$ and $W$.  We can construct such a $\sY$ as follows:  Starting with the projective variety $Y$ over $\Spec(K(Z))$ we can define $Y$ as a subvariety of $\mathbb{P}_{K(Z)}^n$ defined by a finitely generated ideal $I\subset K(Z)[x_0,...,x_n]$.   We can find an affine open subset $U$ of $Z$ for which every coefficient of some generating set of $I$ is contained in $k[U]$.  Then $I$ also defines a variety $\sY_U\subset \mathbb{P}_U^n$ whose generic fiber is $Y$.  Finally we can take the closure of $\sY_U$ in $\mathbb{P}_Z^n$.
	
	The existence of the factorization $W\to X\to \Spec(k(Z))$ means that after shrinking $Z$ we may assume that there is a factorization $\phi:\sW\to\sX$ of $\sW\to Z$, and similarly $\sW\to\sY\to Z$.  In particular, the locus on which $(\sW,\sB_{\sW})$ is not snc is closed in $\sW$ and does not intersect the generic fiber of $\sW\to Z$.  Therefore there is an open subset of $Z$ over which $(\sW,\sB_{\sW})$ is snc, and hence after again shrinking $Z$, we may assume that $\sW\to \sX$ is a log resolution of $(\sX,\sB)$.  Furthermore, after further shrinking $Z$ we may assume that every exceptional divisor of $\phi$ is horizontal over $Z$ and hence intersects the generic fiber.  Denote the resulting open subset of the original $Z$  by $U$.
	
	We claim that now $(\sY,\sB)$ is a log minimal model for $(\sX,\sB)$ over $U$.  Firstly, the part of the definition about discrepancies holds because it is enough to check the discrepancies of exceptional divisors which appear on the log resolution $\sW\to \sX$.  By construction these can all be  calculated on $W\to X$, with the required inequality following from the definition of log minimal model on $X$.
	Furthermore,by \autoref{thm:bpf} $K_Y+B_Y$ is semi-ample.  Therefore there is an open subset $U\subset Z$ such that $K_{\sY}+\sB_{\sY}$ is semi-ample (and hence nef) over $U$, and hence $(\sY,\sB_{\sY})$ is a good log minimal model of $(\sX,\sB)$ over   $U$.
	\end{proof}

\begin{proof}[Proof of \autoref{cor:flip}]
	By \autoref{cor:glmm}, $(\sX,\sB)$ has a good log minimal model$/Y$, over an open subset of $Z$.  The unique log canonical model is the flip.
	\end{proof}

\begin{proof}[Proof of \autoref{cor:finite_generation}]
	First suppose that $(X,B)$ is a quasi-projective variety over a field $k$.  We show that the statement can be reduced to the case of an $F$-finite field.  By \autoref{lem:model} and \autoref{lem:klt-descends}, there is an $F$-finite subfield $k_0\subset k$ and a klt pair $(X_{k_0},B_{k_0})$ and divisor $D_0$ such that $(X,B)=(X_{k_0},B_{k_0})\otimes_{k_0} k$ and $D=D_0\otimes_{k_0}k$.  Firstly, by \cite[Ex 90]{Kollar_exercises}, the required finite generation is equivalent to the existence of a proper, small birational morphism $f:X^+\to X$ such that if $D^+$ is the birational transform of $D$ on $X^+$, then $D^+$ is $\bQ$-Cartier and $f$-ample.  The statement for $X$ and $D$ follow from that for  $X_{k_0}$ and $D_0$ since the conditions small, proper, birational, $\bQ$-Cartier and $f$-ample are all preserved under arbitrary field extension.  
	Now the quasi-projective version over an $F$-finite field is a formal consequence of the (relative) LMMP applied to a projective resolution of singularities, see \cite[92-109]{Kollar_exercises}.  
	
	It remains to prove the case where $(X,B)$ is the localization at a codimension $3$ point of a higher dimensional variety, by reducing to the quasi-projective case.
	Let $(\sX,\sB)$ be a quasi-projective variety of which $(X,B)$ is the localization at the generic point of the closed subvariety $\sZ$ of codimension $3$. Take a pencil of hyperplanes on $\sX$ which restricts to a pencil on $\sZ$.  This gives a rational map $f:\sX\dashrightarrow \mathbb{P}^1$ for which $\sZ\dashrightarrow\mathbb{P}^1$ is dominant.  Let $\sG\subseteq \sX\times \mathbb{P}^1$ 
	be the normalization of the closure of the graph of $f$.  That gives a birational morphism $\phi: \sG\to \sX$ which is an isomorphism near the generic point of $\sZ$ and such that $g:\sG\to \mathbb{P}^1$ is a morphism.  If $\sB_{\sG}$ is the strict transform of $\sB$ on $\sG$, then localizing $(\sG,\sB_{\sG})$ at the generic point of $\sZ$ also gives $(X,B)$.  The same is true of the generic fiber $\sG_{\xi}$ of $\sG\to \mathbb{P}^1$.  Continuing this process, we can eventually assume that the generic fiber has dimension $3$.
	Possibly shrinking $\sG_{\xi}$ around $\sZ_{\xi}$, we may assume that $(\sG_{\xi},\sB_{\sG_{\xi}})$ is klt near the generic point $\eta$ of $\sZ_{\xi}$.	
	 and hence $\oplus\sO_{\sG_{\xi}}(m\phi^*D)$ is finitely generated by localizing the quasi-projective case dealt with earlier.
	\end{proof}




\section{Appendix: Seshadri constants over arbitrary fields}\label{sec:seshadri}

\begin{definition}\cite[4.3.4]{Ful98}\cite[Lemma 5.1.10]{Laz04a} Let $X$ be a variety over an arbitrary field $k$ and $x\in X$ a closed point. Let $\mu:\widetilde{X}\to X$ be the blowup of $X$ at $x$ and $E=\Proj\oplus_{l\>0}\mathfrak{m}_x^l/\mathfrak{m}_x^{l+1}$ the exceptional divisor. The $k$-multiplicity of $X$ at the point $x$ is defined as
	\[\mult_{x/k} X:=(-1)^{\dim X+1} {\left(E^{^{\dim X}}\right)_k}, \] 
	where $(E^{\dim X})_k$ represents the self intersection $E$ over the field $k$.
\end{definition}

\begin{remark}
	The definition of multiplicity  $\mult_x X$ in \cite{Ful98} normalises for the choice of ground field, that is $\mult_{x/k}X=[k(x):k]\mult_x X$.  We use the above definition because it turns out to be the most convenient choice in the definition of Seshadri constants.
	\end{remark}

Next we reinterpret some statements from \cite{Ful98} in this notation.

\begin{lemma}\cite[Example 4.3.5(d)]{Ful98}\label{lem:multiplicity-vs-regularity}
	Let $X$ be a variety over an arbitrary field $k$ and $x\in X$ a closed point. Then $\mult_{x/k} X=[k(x):k]$ if and only if the local ring $\mcO_{X, x}$ is a regular local ring.
\end{lemma}

\begin{lemma}\cite[Example 4.3.9]{Ful98}\label{lem:two-different-multiplicity}
Let $X$ be a variety of dimension at least $2$ over an arbitrary field $k$. Let $D\>0$ be an effective Cartier divisor on $X$ and $x\in X$ a regular closed point, i.e. $\mcO_{X, x}$ is a regular local ring. Let $\vphi$ be a local equation of $D$ and $m_x$ the maximal ideal of $\mcO_{X, x}$. If $d\>0$ is the largest integer such that $\vphi\in m_x^d$, then 
\[\mult_{x/k} D=d[k(x):k] \quad \mbox{ and }\quad \mu^*D=\widetilde{D}+dE, \]
where $\mu:\widetilde{X}\to X$ is the blowup of $X$ at $x$, $E=\Proj\oplus_{l\>0}\mathfrak{m}_x^l/\mathfrak{m}_x^{l+1}$ the exceptional divisor and $\widetilde{D}$ is the strict transform of $D$.
\end{lemma}

\begin{remark}\label{rmk:curve-multiplicities}
	If $C\subset X$ is a curve passing through $x$ and $\tilde{C}$ is the strict transform of $C$ on $\tilde{X}$, then it follows from the definition that $\mult_{x/k} C={E\cdot_k \tilde{C}}.$
\end{remark}

\begin{corollary}\label{cor:intersection-vs-multiplicities}
	Let $X$ be a normal projective variety over $k$. Let $x\in X$ be a regular closed point, $D\>0$ an effective Cartier divisor and $C$ a curve on $X$, each passing though $x$, such that $C$ is not contained in the support of $D$. 
	Then $$D\cdot_k C\>\frac{1}{[k(x):k]}(\mult_{x/k} D)(\mult_{x/k} C).$$
\end{corollary}

\begin{proof}
	Let $\mu:\widetilde{X}\to X$ be the blowup of $X$ at $x$ and $E=\Proj\oplus_{l\>0}\mathfrak{m}_x^l/\mathfrak{m}_x^{l+1}$ the exceptional divisor. Let $\widetilde{D}$ and $\widetilde{C}$ be the strict transform of $D$ and $C$, respectively. Then by  \autoref{lem:two-different-multiplicity} and  \autoref{rmk:curve-multiplicities} it follows that $\mu^*D=\widetilde{D}+\frac{\mult_{x/k} D}{[k(x):k]}E$, and $E\cdot_k\widetilde{C}=\mult_{x/k} C$. Thus $D\cdot_k C=\mu^*D\cdot_k\widetilde{C}=(\widetilde{D}+\frac{\mult_{x/k} D}{[k(x):k]}E)\cdot_k\widetilde{C}\>\frac{\mult_{x/k} D}{[k(x):k]}(E\cdot_k\widetilde{C})=\frac{1}{[k(x):k]}(\mult_{x/k} D)(\mult_{x/k} C)$. 
\end{proof}

\begin{definition}\label{def:seshadri-constant}
	Let $X$ be a projective variety over a field $k$ and $L$ a nef divisor on $X$. Fix a closed point $x\in X$ and let $\mu:X'\to X$ be the blowup of $X$ at $x$ with exceptional divisor $E$. The Seshadri constant 
	\[\ve(X, L; x)=\ve(L; x) \]
	of $L$ at $x$ is the defined to be the real number 
	\[\ve(L; x)=\max\{\ve\>0: \mu^*L-\ve E \mbox{ is nef }\}.\]
\end{definition}

\begin{lemma}\label{lem:2nd-definition-of-seshadri-constant}
	In the situation of \autoref{def:seshadri-constant} one has
	\[\ve(L; x)=\inf_{x\in C\subset X}\left\{\frac{(L\cdot_k C)}{\mult_{x/k} C}\right\}, \]
	where the infimum is taken over all integral curves $C\subset X$ passing through $x$. 	
\end{lemma}

\begin{proof}
	Since $-E$ is $\mu$-ample, to establish the nefness of $\mu^*L-\ve E$ on $X'$ it is enough to intersect with strict transforms of the curves on $X$. Let $C\subset X$ be an integral curve passing through $x$ and $C'$ its strict transform on $X'$. Then
	\[(\mu^*L-\ve E) \mbox{ is nef } \iff (\mu^*L-\ve E)\cdot_k C'\>0 \]
	for every such curve $C$. On the other hand by  \autoref{rmk:curve-multiplicities}
	\[E\cdot_k C'=\mult_{x/k} C. \]
	Therefore
	\[(\mu^*L-\ve E) \mbox{ is nef } \iff \ve\< \frac{(L\cdot_k C)}{\mult_{x/k} C}. \]
	In particular, $\dis \max\{\ve\>0: \mu^*L-\ve E \mbox{ is nef}\}=\inf_{x\in C\subset X}\left\{\frac{(L\cdot_k C)}{\mult_{x/k} C}\right\}$.
\end{proof}

\begin{lemma}\label{lem:vanishing-result}
	Let $X$ be a variety over an arbitrary field $k$ and $x\in X$ a regular closed point of $X$ with ideal sheaf $\mfm=\mfm_x\in\mcO_X$. Consider the blowup
	\[\mu:X'=\Bl_x(X)\to X \]
	of $X$ at $x$, with exceptional divisor $E\subset X'$. Then for every integer $a\>0$
	\[R^j\mu_*\mcO_{X'}(-aE)=0 \mbox{ when } j>0.\]
	In particular,
	\[H^i(X', \mcO_{X'}(\mu^*L-aE))=H^i(X, \mcO_X(L)\otimes\mfm^a)\]
	for every $i\>0$ and every Cartier divisor $L$ on $X$.
\end{lemma}
\begin{proof}
	The proof follows similarly as in the proof of \cite[Lemma 4.3.16]{Laz04a}.
\end{proof}~\\

\begin{definition}
	Let $X$ be a projective variety over $k$. Let $x\in X$ be a regular closed point, and $L$ a line bundle on $X$. We say that $|L|$ \emph{separates $s$-jets} if the natural morphism
	\[\xymatrixcolsep{3pc}\xymatrix{H^0(X, L)\ar@{->>}[r] & H^0(X, L\otimes\mcO_X/\mathfrak{m}^{s+1}_x)}\] 
	is surjective.\\
	We denote by $s(L; x)$ the largest non-negative integer $s$ such that $|L|$ separates $s$-jets at $x$. If $x$ is a base-point of $|L|$, then we define $s(L; x)=-1$. 	
\end{definition}

\begin{proposition}
	Let $L$ be an ample line bundle on a projective variety $X$, and $x\in X$ a regular closed point. Then
	\[ \ve(X, L; x)=\lim_{m\to\infty}\frac{s(mL; x)}{m}.\]
\end{proposition}

\begin{proof}
	Write $\ve=\ve(L; x)$ and $s_m=s(mL; x)$. We must prove the inequalities
	\[\ve\>\limsup \frac{s_m}{m}\>\liminf\frac{s_m}{m}\>\ve. \]
	
	For the left inequality, we will show that $\ve\>\frac{s_m}{m}$ for every $m$. Let $C$ be an integral curve passing through $x$. Since $|mL|$ separates $s_m$-jets at $x$, by  \autoref{lem:two-different-multiplicity} we can find a divisor $F_m\in|mL|$	with $\mult_{x/k} F_m\>s_m[k(x):k]$ and $C\nsubseteq F_m$. Then
	\begin{equation*}
	\begin{split}
	m(L\cdot_k C) &=F_m\cdot_k C\\
	&\> \frac{1}{[k(x):k]}(\mult_{x/k} F_m)(\mult_{x/k} C) \mbox{  [by \autoref{cor:intersection-vs-multiplicities}]}\\
	&\> s_m\cdot\mult_{x/k} C.
	\end{split}
	\end{equation*}	
	Then by  \autoref{lem:2nd-definition-of-seshadri-constant}	$\ve\>\frac{s_m}{m}$. In particular $\ve\>\limsup\frac{s_m}{m}$.
	
	The other nontrivial inequality holds by the same argument as the relevant part of  \cite[5.1.17]{Laz04a}, using the characterization of the Seshadri constant in terms of nefness on blowups.
\end{proof}

\begin{definition}\cite[Definition 2.4]{mustata_frobenius_2012}
	Let $X$ be a variety over an $F$-finite field $k$, $x\in X$ a regular closed point and $L$ a line bundle. 
	We say that $L$ separates $p^e$-Frobenius jets at $x$ if the map
	$$H^0(X,L)\to H^0(X,L\otimes\sO_X/\frak{m}_x^{[p^e]})$$ is surjective.
	Let $s_F(L^m;x)$ be the largest $e$ such that $L^m$ separates $p^e$-Frobenius jets at $x$.  (If there is no such $e$ then $s_F=0$.)
	 We define the Frobeinius-Seshadri constant at $x$ to be 
	$$\epsilon_F(L;Z):=\sup_{m\geq 1} \frac{p^{s_F(L^m;Z)}-1}{m}$$
	\end{definition}

The proofs in \cite{mustata_frobenius_2012} go through when the assumption of an algebraically closed base field is replaced by that of an $F$-finite field, and if ``smooth'' is replace by ``regular'' throughout, so we have the following:

\begin{theorem}\cite[Theorem 3.1]{mustata_frobenius_2012}
	Let $L$ be an ample line bundle on regular variety $X$ over $F$-finite field $k$.  
	If $\epsilon_F(L,x)>1$ then the restriction map 
	$$H^0(X,\omega_X\otimes L)\to H^0(X,\omega_X\otimes L\otimes \sO_X/\frak m_x)$$ is surjective.
\end{theorem}

\begin{proposition}\cite[Proposition 2.12]{mustata_frobenius_2012}
	Let $L$ be an ample line bundle on a projective variety $X$ over an $F$-finite field $k$, with a regular point $x$.  Then we have
	$$\frac{\epsilon(L,x)}{n}\leq \epsilon_F(L;x)\leq \epsilon(L;x)$$
	\end{proposition} 

\bibliographystyle{hep}
\bibliography{references.bib}

\end{document}